\documentclass[letterpaper,11pt]{amsart}

\usepackage{amssymb}
\usepackage[T1]{fontenc}
\usepackage[latin9]{inputenc}
\usepackage[linktocpage = true]{hyperref}
\usepackage{color}
\usepackage{array}
\usepackage{float}
\usepackage{mathrsfs}
\usepackage{mathtools}
\usepackage{amsthm}
\usepackage{amstext}
\usepackage{stmaryrd}
\usepackage{graphicx}
\usepackage{pgfplots}
\usepackage{tikz-cd}

\theoremstyle{plain}
\newtheorem{thm}{\protect\theoremname}
\newtheorem{thmletter}{Main Theorem}

\theoremstyle{plain}
\newtheorem{prop}[thm]{\protect\propositionname}
\theoremstyle{plain}
\newtheorem{cor}[thm]{\protect\corollaryname}
\theoremstyle{plain}
\newtheorem{lem}[thm]{\protect\lemmaname}
\theoremstyle{definition}

\theoremstyle{definition}
\newtheorem{defn}[thm]{\protect\definitionname}
\theoremstyle{definition}
\newtheorem{notation}[thm]{\protect\notationname}

\makeatletter
\def\ps@pprintTitle{%
  \let\@oddhead\@empty
  \let\@evenhead\@empty
  \let\@oddfoot\@empty
  \let\@evenfoot\@oddfoot
}
\makeatother

\AtBeginDocument{
	
}

\makeatother

  \providecommand{\corollaryname}{Corollary}
  \providecommand{\examplename}{Example}
  \providecommand{\lemmaname}{Lemma}
  \providecommand{\propositionname}{Proposition}
  \providecommand{\theoremname}{Theorem}
  \providecommand{\definitionname}{Definition}
  \providecommand{\notationname}{Notation}

\newcommand{\mult}{\operatorname{mult}}

\newcommand{\Kos}{\operatorname{Kos}}
\newcommand{\gr}{\operatorname{gr}}
\newcommand{\lev}{\operatorname{lev}}
\newcommand{\Ann}{\operatorname{Ann}}

\newcommand{\Spec}{\operatorname{Spec}}

\newcommand{\mh}{\operatorname{mh}} 
\newcommand{\im}{\operatorname{im}}
\newcommand{\id}{\operatorname{id}}
\newcommand{\Germs}{\mathbf{Germs}}
\newcommand{\Sets}{\mathbf{Sets}}
\newcommand{\Def}{\mathbf{Def}}
\newcommand{\RNum}[1]{\uppercase\expandafter{\romannumeral #1\relax}}

\begin{document}

	\title{On the $k$-th Tjurina number of weighted homogeneous singularities}
	\author{Chuangqiang Hu}
	\address{Sun Yat-Sen University, School of Mathematical, Guangzhou, 510275, P. R. China} 
	\email{huchq@mail2.sysu.edu.cn}
	\thanks{Hu was supported by NSFC Grant 12441107.}
	
	\author{Stephen S.-T.~Yau}
	\address{Beijing Institute of Mathematical Sciences and Applications,  Beijing 101408, P. R. China} 
	\email{yau@uic.edu}
	\thanks{Corresponding author: Stephen~S.-T. Yau\\
	Yau was supported by Tsinghua University start-up fund and Tsinghua University Education Foundation fund (042202008).}
	
	\author{Huaiqing~Zuo}
	\address{Department of Mathematical Sciences, Tsinghua University, Beijing 100084, P. R. China}
	% \email{hqzuo@mail.tsinghua.edu.cn}
	\email{hqzuo@mail.tsinghua.edu.cn}
	\thanks{Zuo was supported by NSFC Grant 12271280.}

	\keywords{Deformation Functor, Singularity, Koszul Complex}
\subjclass[2020]{Primary 14B05 Secondary 32S05}

	\begin{abstract}
		Let $ (X,0) $ denote an isolated singularity defined by a weighted homogeneous polynomial $ f $. Let $ \mathcal{O}$ be the local algebra of holomorphic function germs at the origin, with the maximal ideal $m $. We study the $k$-th Tjurina algebra, defined by $ A_k(f): = \mathcal{O} / \left( f , m^k J(f) \right) $, where $J(f)$ denotes the Jacobian ideal of $ f $. The zeroth Tjurina algebra is well known to represent the tangent space of the base space of the semi-universal deformation of $(X, 0)$. Motivated by this observation, we explore the deformation of $(X,0)$ with respect to a fixed $k$-residue point. We show that the tangent space of the corresponding deformation functor is a subspace of the $k$-th Tjurina algebra.
		Explicit calculation of the $k$-th Tjurina numbers, which correspond to the dimensions of the $k$-th Tjurina algebras, plays a crucial role in understanding these deformations. According to the results of Milnor and Orlik, the zeroth Tjurina number can be expressed explicitly in terms of the weights of the variables in $f$. However, we observe that for values of $k$ exceeding the multiplicity of $X$, the $k$-th Tjurina number becomes more intricate and is not solely determined by the weights of the variables.
		In this paper, we introduce a novel complex derived from the classical Koszul complex and obtain a computable formula for the $k$-th Tjurina numbers for all $ k \geqslant 0 $. As an application, we calculate the $k$-th Tjurina numbers for all weighted homogeneous singularities in three variables.
\end{abstract}

\maketitle

\tableofcontents
%% main 
\section{Introduction}
Let $ (\mathbb{C}^n,0) $ denote a germ of an $n$-dimensional complex space located at the origin.  We are interested in studying a hypersurface singularity, which is defined by a complex analytic function $f = f(x_1,\ldots,x_n)$ having an
isolated critical point at the origin. The locus $V(f)$ is defined as the set of points $(x_1,\ldots, x_n) \in \mathbb{C}^n$ satisfying $f(x_1,\ldots, x_n) = 0$. 
To analyze the algebraic property of these singularities, we introduce some mathematical concepts. Let $\mathcal{O} = \mathbb{C}[[x_1,\ldots,x_n]]$ be the formal power series ring in variables $x_1,\ldots , x_n$. For $i = 1, \ldots, n $, we denote by $f_i = \partial f / \partial x_i $ the partial derivatives of $f$. The Jacobian ideal $ J(f)$ of $ V(f)$ is generated by $f_1,\ldots,f_n \in \mathcal{O}$.  We assume that $V(f)$ has an isolated singularity at the origin.

The moduli algebra associated to $V(f) $ is defined as a $\mathbb{C}$-algebra: 
\[
	A(f) := \mathcal{O} / (f ,J(f)),
\]
while the Milnor algebra associated to $V(f) $ is given by 
\[
	M(f) := \mathcal{O} / (J(f)). 
\]
It is well known
that the algebra $A(V)$ is finite-dimensional if and only if the germ
$V(f)$ has an isolated singularity (see e.g. \cite{greuel2007introduction}).
By studying the moduli algebra $A(f)$ we can gain a deep understanding of the hypersurface singularity $V(f)$ and its local behavior. The result of Mather and Yau \cite{mather1982classification} states that the biholomorphic equivalence class of an isolated hypersurface singularity is determined by its moduli algebra. 

In the realm of singularity theory, the dimension of the moduli algebra $A(f)$, symbolized as $\tau_0$,is an important invariant known as the Tjurina number. This quantity serves as a quantitative measure of the singularity's complexity, offering valuable insights into its local geometric and topological characteristics. Notably, the infinitesimal deformation of $(V(f), 0)$ is unobstructed. Consequently, the associated semi-universal space of $(V(f), 0)$ exists and is formally smooth. Intriguingly, the tangent space of this semi-universal space coincides with the moduli algebra $A(f)$. Therefore, $\tau_0$ can be precisely interpreted as the dimension of the tangent space of the semi-universal space, providing a crucial link between algebraic and geometric aspects of the singularity.

Similarly, the dimension $\mu_0$ of the Milnor algebra, called the Milnor number, plays a central role in singularity theory. It provides indispensable information regarding the topological structure and classification of singularities. As established by Milnor in \cite{milnor2016singular}, the link of $V(f)$ has the homotopy type of a bouquet of spheres, and remarkably, the number of spheres in this bouquet is equal to the Milnor number of $V(f)$. This relationship enhances our understanding of the singularity's topology and provides a powerful tool for its classification and further analysis. The Milnor number, together with the Tjurina number, forms the backbone of many investigations in singularity theory, enabling researchers to dissect the complex behavior of singularities from multiple perspectives.

In the extensive landscape of isolated hypersurface singularities, weighted homogeneous singularities have consistently attracted significant attention from researchers. These singularities possess unique properties that make them a fascinating subject of study within the realm of singularity theory. 
Recall that a polynomial $f ( x_1 , \cdots, x_n )$ is weighted homogeneous of a specific type 
$(w_1 , w_2 , \cdots, w_n )$, where $ w_1 , w_2 , \cdots, w_n $ are fixed positive rational numbers, if it can be expressed as a linear combination of monomials $x_1^{i_1} x_2^{i_2} \cdots x_n^{i_n} $
such that 
\[ i_1 w_1 + i_2 w_2 + \cdots + i_n w_n = W 
\] for some constant $W$.
A natural
question is when $V(f) $ is defined by a weighted homogeneous polynomial up to biholomorphic change of coordinates. 
Saito \cite{saito1971quasihomogene} solved this question. According
to Saito's theorem, $ V(f) $ is equivalent to a weighted homogeneous singularity after a biholomorphic change of coordinates if and only if the Milnor number coincides with the Tjurina number. 

The Milnor number exhibits a significant connection with the geometric genus $p_g$, especially in the context where the polynomial $f$ is weighted homogeneous. This relationship has been the subject of much investigation and speculation within the realm of singularity theory.

In 1978, Durfee proposed an interesting conjecture that  \[ \mu_0 \geqslant (n + 1)! \cdot p_g . \]  This conjecture seeks to establish a fundamental relationship between these two important invariants, shedding light on the  geometric and algebraic structure of weighted homogeneous singularities.

Subsequently,  Xu and Yau \cite{xu1993durfee} made significant progress in this regard for the two-dimensional case of weighted homogeneous singularities. They  proved that the inequality
\[  \mu_0- m_0 + 1 \geqslant 6 \cdot p_g \]
is valid, where $m_0$ represents the multiplicity of the polynomial $f$. This result provided a more refined understanding of the connection among the Milnor number, the multiplicity, and the geometric genus in the two-dimensional setting, offering valuable insights into the specific characteristics of these singularities in this dimension.

For the three-dimensional weighted-homogeneous case, as demonstrated in \cite{lin2004classification}, another inequality has been established:
\[
\mu_0-(2 \cdot m_0^3-5 \cdot m_0 ^2+2 \cdot m_0 +1) \geqslant 4! \cdot p_g. 
\]
This inequality  extends our knowledge of the relationship between these invariants in a higher-dimensional context. It showcases how the Milnor number, adjusted by a specific expression involving the multiplicity $m_0$, relates to the geometric genus in the three-dimensional scenario. Overall, these results contribute to a more comprehensive understanding of the intricate interplay between these crucial invariants in weighted homogeneous singularities across different dimensions, helping researchers to better analyze and classify such singularities based on their algebraic and geometric properties. 

In the literature, the derivation Lie algebra of $A(f)$ is called the Yau algebra, and has been the extensively studied. This algebra is of interest due to its interesting properties and its relevance within the context of singularity theory. It is shown in \cite{seeley1990variation} that Yau algebra is finite dimensional and its connection to moduli spaces of singularities is quite substantial. Since the 1980s, Yau and his collaborators have embarked on a systematic study of the Yau algebra and its generalizations. Their work, including \cite{yau1983continuous}, \cite{seeley1990variation}, \cite{yau1991solvability}, and others, has greatly contributed  to our understanding of this algebra.

In the present paper, our focus will be placed on the generalized version of the moduli algebra, constructed via the $m$-filtration of the Jacobian ideal $J(f)$, where $m$ represents the maximal ideal of $\mathcal{O}$. Following \cite{GreuelPham2017, greuel2007introduction}, the $k$-th Tjurina algebra and the $k$-th Milnor algebra of the isolated hypersurface singularity  
$V(f)$ are defined respectively as 
\[
	A_{k}(f) :=  \mathcal{O} / \left( f, m^k J(f) \right),
\]
and
\[
	M_{k}(f) := \mathcal{O} / m^k J(f).
\]
The dimensions of these algebras are fundamental invariants of the singularity $V(f)$.  Specifically, we call $\tau_k = \tau_k(V(f))$  the $k$-th Tjurina number, and $\mu_k = \mu_k(V(f))$  the $k$-th Milnor number, adopting the terminology of \cite{hussain2023k}.

%  In the present paper, our focus will be placed on the generalized version of the moduli algebra, constructed via the $m$-filtration of the Jacobian ideal $J(f)$, where $m$ represents the maximal ideal of $\mathcal{O}$. Following \cite{GreuelPham2017, greuel2007introduction, hussain2023k}, the $k$-th Tjurina algebra and the $k$-th Milnor algebra of the isolated hypersurface singularity  
% $V(f)$ are defined respectively as 
% \[
% 	A_{k}(f) :=  \mathcal{O} / \left( f, m^k J(f) \right),
% \]
% and
% \[
% 	M_{k}(f) := \mathcal{O} / m^k J(f).
% \]
% The dimensions of these algebras are fundamental invariants of the singularity $V(f)$.  Specifically, we call $\tau_k = \tau_k(V(f))$  the $k$-th Tjurina number, and $\mu_k = \mu_k(V(f))$  the $k$-th Milnor number.

It is worth noting that the analog Mather-Yau theorem concerning the $k$-th Tjurina algebra can be located in different references depending on the field under consideration. For the complex number field, one can refer to \cite{greuel2007introduction}, while for fields of positive characteristic, relevant information can be found in \cite{GreuelPham2017}. This theorem plays an important role in establishing connections and understanding the properties related to these generalized algebraic structures in different algebraic settings, further enriching our exploration of the isolated hypersurface singularities and their associated algebraic invariants.

Motivated by the well-known classical deformations of $(V(f),0)$, an interesting question arises as to whether the algebras $A_k(f)$ bear any relationship with the deformations of $(V(f),0)$. To explore this possibility, we introduce a specific deformation functor, denoted as
\begin{equation*}
	\Def_{k}^{V(f)}:  \; \Germs  \to \Sets,
\end{equation*}
which plays a crucial role in our investigation. This functor maps each germ $\mathcal{T}$ to the collection of equivalence classes of $k$-pointed deformations over $\mathcal{T}$.

To clarify what a $k$-pointed deformation entails, it can be roughly described as a chain of successive morphisms in the form $\mathcal{P} \to \mathcal{X} \to \mathcal{T}$. There are specific requirements for these morphisms: $\mathcal{X} \to \mathcal{T}$ must be a classical deformation of $(V(f),0)$, while $\mathcal{P} \to \mathcal{T}$ is required to be a trivial deformation of the fat point $\Spec \mathcal{O}/(f, m^k)$. Through this construction, we are able to study the connection between the algebraic structure of $A_k(f)$ and the deformations in a more systematic way. It turns out that there is a significant relationship between the tangent space of the deformation functor $\Def_{k}^{V(f)}$ and $A_k(f)$. Specifically, the tangent space of $\Def_{k}^{V(f)}$ is dominated by $A_k(f)$. This implies that the $k$-th Tjurina numbers are deeply intertwined with the local structure of $(V(f),0)$.

The following theorem further elaborates on this relationship:
\begin{thmletter}
	Suppose that $ (V(f),0) \subseteq (\mathbb{C}^n ,0) $ represents an isolated hypersurface singularity.  In this context, the tangent space $T \Def_{k}^{V(f)} $ of the deformation functor $\Def_{k}^{V(f)}$ is isomorphic to 
	$	  {(f,m^k) }/ { \left( f ,m^k J(f)  \right) }.
	$
	Therefore, we have the equality
	 \[ \dim T \Def_{k}^{V(f)} = \tau_k - \dim {\mathcal{O}}/{(f, m^k)} . \] 
\end{thmletter}
For the particular cases where  $k =0 $ or $k =1$, this result is already well established within the existing literature. Specifically, one can refer to Lemma 2.5 in \cite{hirsch2006deformations} for a detailed treatment of these cases. 

Subsequently, our attention shifts towards the computation of the $k$-th Milnor (and Tjurina) numbers. The Milnor-Orlik theorem, as presented in \cite{Milnor1970}, played a significant role in formulating the Milnor number for weighted homogeneous singularities by leveraging the weights of the polynomial $f$.
\begin{thm}[Milnor-Orlik]
	Let $f (x_1, \cdots, x_n )$ be a weighted homogeneous polynomial of type $(w_1,\ldots, w_n)$ along with a total weight $W$, and under the assumption that $V(f)$ represents an isolated singularity at the origin. In this context, the Milnor (and equivalently, the Tjurina) number is given by the expression
\[
	\tau_0 = \mu_0 = \prod_{i=1}^{n} \left(\frac{W}{w_i}- 1\right).
\]
\end{thm}
The question regarding the generalization of the formula from the context of the traditional Milnor and Tjurina numbers to the $k$-th Tjurina numbers is both significant and naturally arises in the study of these invariants. As mentioned earlier, such a generalization would greatly enhance our comprehension of how these invariants relate to the weighted homogeneous structure in a broader and more comprehensive manner.

The work carried out in \cite{hussain2023k} has been quite illuminating, as the authors have successfully computed the $k$-th Tjurina numbers and the $k$-Milnor numbers specifically for two-dimensional homogeneous singularities. This achievement has paved the way for further investigations and observations.

Motivated by their work, we have arrived at some notable observations. In the case where $k \leqslant m_0$, we can make the interesting conclusion that the $k$-th Tjurina numbers depend solely on the weights, mirroring the situation in the Milnor-Orlik theorem to some extent. This is formalized in the following theorem:

\begin{thmletter}\label{thm:muIntro}
	Let $V(f)$ be a weighted homogeneous singularity with Milnor number $\mu_0$. Here, we denote by $c$ the number of weights of $f$ that attain the maximal value. The explicit formulas for $\mu_k$ and $\tau_k$ are as follows:

\begin{align}
	\mu_k & = \mu_0 + n \binom{k-1+n}{n} \text{ for } k < m_0; \label{eq:muk} \\
	\mu_{m_0} & = \mu_0 + n \binom{m_0-1+n}{n}  -\frac{1}{2} c ( 2 n -c- 1);  \label{eq:mum0} \\
	\tau_k & = \mu_0 + n \binom{k-1+n}{n}  - \binom{k-2+n}{n} \text{ for } k < m_0; \label{eq:tauk}\\
	\tau_{m_0} & = \mu_0 + n \binom{m_0-1+n}{n} - \binom{m_0-2+n}{n}  -\frac{1}{2} c ( 2 n -c- 1)  \label{eq:taum0}.
\end{align} 
\end{thmletter}

These formulas offer precise ways to calculate the $k$-th Milnor and Tjurina numbers under the specified conditions and provide valuable insights into their dependence on the weights and other relevant parameters.

However, it is important to note that the situation changes when we consider the $k$-th Tjurina numbers with $k \geqslant m_0$. In such cases, these numbers are not simply determined by the weights alone. This indicates that there are additional factors or complexities that come into play when dealing with larger values of $k$, and further research would be needed to fully understand and characterize the behavior of these invariants in such scenarios. 
As the $k$-th Tjurina numbers are more complicated to formulate, we  describe the both numbers for all $ k \geqslant 0 $ by means of Hilbert-Poincar\'{e} series
\[
	\mathbb{A}_f(t) := \sum_{k=0}^\infty \tau_k  t^k 
	\text{ and }
	\mathbb{M}_f(t) := \sum_{k=0}^{\infty} \mu_k t^k 
\] 
respectively. 
The main goal of this paper is to derive the formulas of these series which can be viewed as a natural continuation of Milnor-Orlik theorem.
\begin{thmletter}\label{thm:AtMtIntro}
	Let $ f $ be a weighted homogeneous polynomial and assume that $V(f) $ is an isolated hypersurface singularity. 
	Denote by $ m_i $ the multiplicity of $ f_i = \frac{\partial f }{\partial x_i} $ and set $ m_{i,j} = \min \{m_i, m_j \}$. 
	Then the series $ \mathbb{M}_f(t)  $ and $ \mathbb{A}_f (t) $ are given by the following formulas:
	\begin{equation}\label{equ:Mt_intro}
	\mathbb{M}_f(t) = \frac{t}{(1-t)^{n+1} } \left ({n -\sum_{i<j}  t^{m_{i,j}}} \right ) + \frac{\mu_0 + \mathbb{Z}_{\infty}(t) \cdot t  }{1-t} + \sum_{i=1}^{k} \frac{(t-t^{L_i+1})}{1-t} \mathbb{H}_{L_i} (t)   
	\end{equation}
	and  
	\begin{equation}\label{equ:At_intro}
		\mathbb{A}_f(t) = \frac{t}{(1-t)^{n+1} } \left ({n-t-\sum_{i<j}  t^{m_{i,j}}} \right ) + \frac{\mu_0 + \mathbb{Z}_{\infty}(t) \cdot t  }{1-t} + \sum_{i=1}^{k} \frac{(t-t^{L_i+1})}{1-t} \mathbb{H}_{L_i} (t)  ;
	\end{equation}
	where $L_1, \ldots, L_k$ are the gap numbers, 
	$ \mathbb{Z}_{\infty}(t) $ and $ \mathbb{H}_{L_i}(t) $
	are series associated to $ f $. See Section \ref{sec:MainResult} for the detailed definitions.
\end{thmletter}

The key to the proof is to switch from the classical Koszul complex to an appropriate resolution of graded modules of the Jacobian ideal $J(f)$. We discover that the related homology space admits a natural bigraded structure by using the degree and our level filtration. The relative graded modules are the crucial clue to characterize the Hilbert-Poincar\'{e} series properly.  

Another main contribution of the paper is give a demonstration of the formulas 
\eqref{equ:Mt_intro} and \eqref{equ:At_intro} for the complete list of weighted homogeneous singularities in three variables. 
We achieve the explicit formulas for both $ \mathbb{A}_f(t) $ and $ \mathbb{M}_f(t)$, and thus the $k$-th Tjurina numbers and the $k$-th Milnor numbers are understood well. Precisely, we have the following results.
\begin{thmletter}\label{thm:seriesIntro}
	Suppose that the polynomials $f^{(i)}$ with $ i =1 ,\ldots , 7$ form the complete list of weighted homogeneous singularities in three variables.  
	We obtain the series:
	\begin{equation*} 
		\mathbb{M}_f (t) =	\frac{ \mu_0  }{1-t} + \frac{3 t + t \cdot \mathbb{L}_i(t)}{(1-t)^4}
	\end{equation*}
	and  
	\begin{equation*} 
		\mathbb{A}_f(t) =	\frac{ \mu_0  }{1-t} + \frac{3 t -t^2+ t \cdot \mathbb{L}_i(t)}{(1-t)^4} .
	\end{equation*}
	The precise definitions of $ f^{(i)} $ and $ \mathbb{L}_{i}(t) $ are given in Section \ref{sec:Applications}.
\end{thmletter}

The paper is structured as follows.
In Section \ref{sec:Geometric} we introduce the pointed deformations and describe the tangent space in terms of the $ k$-th moduli algebra. 
In Section \ref{sec:Connections}, we reduce the computation of $k$-th Tjurina numbers to the graded module $ J(f) $ by means of Hilbert-Poincar\'{e} series.
In Section \ref{Representation},  We construct the modified version of Koszul complex, and deal with the filtration of the correspondent homology.
Section \ref{sec:MainResult} is devoted to proving our Main Theorem \ref{thm:AtMtIntro} with the help of the new Koszul type complex.
In Section \ref{sec:Low}, we determine the $k$-th Tjurina numbers for the two-dimensional singularities and also discuss the lower bound for three-dimensional singularities. 
Section \ref{sec:Applications} deals with the Hilbert-Poincar\'{e} series of each type of three-dimensional weighted homogeneous singularities in Main Theorem \ref{thm:seriesIntro}. 

%  state the geometric perspective.
\section{Geometric Perspective of $k$-th Tjurina Numbers}\label{sec:Geometric}
\subsection{Deformation of Singularity}
We recall some basic knowledge of deformation theory. Let $ \Germs  $ be the category of local analytic variety germs and $\mathbf{Sets}$ the category of sets.
Let $ (X , p ) $ be an isolated hypersurface singularity in $ \mathbb{C}^n $. 
Given a local germ $ (\mathcal{T},0) $, a deformation over the base $ \mathcal{T} $ means a local germ $ (\mathcal{X} , \iota(p) )$ associated with germ embedding $ \iota:  X \to \mathcal{X} $ and germ projection $ \rho : \mathcal{X} \to \mathcal{T} $ such that the diagram  
\[
	\begin{tikzcd}
		X \arrow{d}[swap]{ } \arrow{r}{\iota}  
		& \mathcal{X}  \ar[d,"\rho"] \\
		0 \arrow{r}[swap]{} &{  \mathcal{T} }
	\end{tikzcd}
\]
is a Cartesian square.
Deformations $ \rho:\mathcal{X} \to 
\mathcal{T} $ and $ \rho': \mathcal{X}' \to \mathcal{T}'$ of $(X,0)$ are said to be equivalent if there exist some isomorphisms 
$ \Psi : \mathcal{X} \to \mathcal{X}' $ and $ \psi :\mathcal{T} \to \mathcal{T}' $ such that the diagram
\[
	\begin{tikzcd}[row sep =tiny]
		{} & 	X\ar{ld}\ar{rd}
		\ar[dd]  & [1.5em] \\
		\mathcal{X}\ar[dd, "\rho"] \ar[rr, "\Psi",crossing over] & {} &\mathcal{X}'\ar[dd, "\rho'"]\\
		{} &  0 \ar{ld}\ar{rd}  & {} \\
		\mathcal{T} \ar[rr, "\psi"]& {} &  \mathcal{T}' 
	\end{tikzcd}
\]
commutes.

\subsection{Pointed Deformations}
We assume that $ (X ,p) $ is an isolated singularity located at the point $p$. Let $ \mathcal{O}_X $ be the local ring of $X$ with the maximal ideal $ m_X $. The fat point $ P$ associated to the closed point $p$ with structure algebra 
\[
	\mathcal{O}_{X,k} :=  \mathcal{O}_X / m_X^k
\]
 is called the $k$-residue fat point of $ X $.
The induced morphism $P \to X$ corresponds to the quotient map  $
	\mathcal{O}_X \to \mathcal{O}_{X,k} 
$.
\begin{defn}
	Let $ P $ be the $k$-residue fat point of singularity $ (X,p) $. The chain of successive morphisms $ \mathcal{P} \to \mathcal{X} \to \mathcal{T} $ together with embeddings $ P \to \mathcal{P}$ and $ X \to \mathcal{X} $ is called a $ k $-pointed deformation of $ X $ over 
	$\mathcal{T} $ if the following conditions hold:  
	\begin{enumerate}
		\item The pullback of the chain $ \mathcal{P} \to \mathcal{X} \to \mathcal{T} $ to $ 0 \in \mathcal{T} $ is identical to $ P \to X \to 0 $. In other words, the two squares of the diagram 
		\[\begin{tikzcd}
			P \ar[d]\ar[r] & \mathcal{P}\ar[d] \\
 			X \ar[d]\ar[r] & 	 \mathcal{X}\ar[d] \\
			0 \ar[r] & \mathcal{T} 
		\end{tikzcd} \]
		are Cartesian.
		\item The morphism $\mathcal{X} \to \mathcal{T} $ is a deformation of $ X$ over $\mathcal{T}$. 
		\item The morphism $ \mathcal{P} \to \mathcal{T} $ is equivalent to a trivial deformation of $ P $ over $ \mathcal{T}$, and $ \mathcal{P} \to \mathcal{X} $ is an embedding. 
	\end{enumerate}
	A $k$-pointed deformation over $ \mathcal{T}$ will be denoted by $ (\mathcal{P} \to \mathcal{X} \to \mathcal{T})$.
\end{defn}
Given two $k$-pointed deformations of $ X $, namely, $ \mathcal{P} \to  \mathcal{X} \to \mathcal{T} $ and $\mathcal{P}'\to \mathcal{X}' \to \mathcal{T}' $, we say they are equivalent if the diagram
\[
	\begin{tikzcd}[row sep =tiny]
		{} & P \ar[ld]\ar[dd] \ar[rd] & [1.5em] \\
		\mathcal{P} \ar[rr, crossing over]\ar[dd] & {} & \mathcal{P}' \ar[dd] \\
		{} & 	X\ar{ld}\ar{rd}\ar{dd}  & {} \\
		\mathcal{X}\ar{dd} \ar[rr, crossing over] & {} &\mathcal{X}'\ar{dd}\\
		{} &  0 \ar{ld}\ar{rd}  & {} \\
		\mathcal{T} \ar[rr, crossing over]& {} &  \mathcal{T}' 
	\end{tikzcd}
\]
commutes.

In particular, when $k = 1$, then the $1$-residue point $ P $ is just the closed point $ p $ and the $1$-pointed deformation of $ X $ coincides with the deformation with sections (see \cite{greuel2007introduction} and \cite{hirsch2006deformations}).

\subsection{Deformation Functor}
For singularity $(X,p)$, we consider the deformation functor
\begin{equation*}
	\Def_{k}^{X}:  \;  \Germs  \to \Sets,
\end{equation*} 
which sends each germ $ \mathcal{T} $ to the collection of
equivalence classes of $k $-pointed deformations $ \mathcal{X} $ over $\mathcal{T}$. 
Let $\mathcal{T}_{\epsilon}$ be the germ with the structure algebra $\mathbb{C}[\epsilon]$ with $ \epsilon^2 = 0$. 
The tangent space of $ \Def_{k}^{X} $ is defined to be the first order $k$-pointed deformation of $ X $, namely 
\[
	T \Def_{k}^{X} = \Def_k^{X} (\mathcal{T}_{\epsilon}).
\]
 We assume from now on that $ X $ is located at the origin of smooth germ $ (\mathbb{C}^n , 0 )$. Let $ \mathcal{X} \to \mathcal{T}$ be the deformation of $X$. 
There exists an unfolding morphism $\mathcal{X} \to \mathbb{C}^n \times \mathcal{T} $. 
Up to coordinate transformations, one may transform the pointed singularity of each fiber $ \mathcal{X}_t $ to the origin of $ \mathbb{C}^n $, where the fat point $ \mathcal{P}_t $ is located at the origin of each fiber. Applying another coordinate transformation if necessary, we may further assume that 
$ \mathcal{P} $ is the trivial deformation that factors through $ \mathbb{C}^n \times \mathcal{T} \to \mathcal{T}$. 

Now we consider the case that $ (X ,p)= (V(f) , 0)$ is a hypersurface singularity in  $  \mathbb{C}^n $.
Recall $\mathcal{O}$ the local algebra of $ (\mathbb{C}^n,0) $ with the maximal ideal $ m $. We have $ \mathcal{O}_X = \mathcal{O}/ (f) $. Then $ m_X = m \cdot \mathcal{O}/ (f) $ is the maximal ideal of $ \mathcal{O}_X$. It follows that the structure of the $k$-residue point $ P $ is given by  $ \mathcal{O}_{f,k} := \mathcal{O} / (m^k , f )$. One may easily check that 
\[
	\dim \mathcal{O}_{f,k} = \begin{cases}
		\binom{n+k-1}{n}  & \text{for $k < m_0 $}; \\
		\binom{n+k-1}{n} - \binom{n+k-1-m_0 }{n} & 	\text{for $k \geqslant m_0 $},
	\end{cases}	
\]
where $ m_0 = \mult(f)$ denotes the multiplicity of $f$. As discussed above, a $k$-pointed deformation over $ \mathcal{T}_\epsilon $ is represented by 
\[
	\mathcal{P} : = P \times \mathcal{T}_\epsilon \to \mathcal{X}, \]
where 
$ \mathcal{X}:= V(f + g \epsilon ) \subseteq \mathbb{C}^n \times \mathcal{T}_\epsilon$ with $ g \in \mathcal{O}$. The corresponding structure morphism of $\mathcal{P} \to \mathbb{C}^n\times \mathcal{T}_{\epsilon} $ is represented by the quotient map 
\begin{equation}\label{eq:quotient}
	\mathcal{O} [\epsilon] \to \mathcal{O} [\epsilon]/ (f,m^k).
\end{equation}
The relation between the $ k$-th Tjurina number and the $k$-pointed deformation is stated in the following theorem.
\begin{thm}
	Suppose that $ (V(f),0) \subseteq (\mathbb{C}^n ,0) $ is an isolated hypersurface singularity.  Then
	\begin{equation}\label{equ:TDef}
		T \Def_{k}^{V(f)} \cong  \frac{(f,m^k) }{ \left( f ,m^k J(f)  \right) }.
	\end{equation}
	Therefore, we have $\dim T \Def_{k}^{V(f)} = \tau_k - \dim \mathcal{O}_{f,k} $. 
\end{thm}
\begin{proof}
	We have already known that the first order the deformation $ \mathcal{X}\to \mathcal{T}_{
		\epsilon}$ of $ X := V(f)$ is equivalent to the hypersurface 
	$ \mathcal{F} = f + g \epsilon $
	of $\mathbb{C}^n \times T_{\epsilon}$ for some (formal) analytic function $g \in \mathcal{O} $. The germ morphism $ \mathcal{P} \to \mathbb{C}^n \times \mathcal{T}_\epsilon $ in this case corresponds to the quotient map
	\[
		\mathcal{O}[\epsilon] \to \mathcal{O}[\epsilon]/(\mathcal{F}, m^k). 	
	\]
	Since we require the residue fat point of fiber $ V(\mathcal{F}_\epsilon)$ is the same with $P$, we have 
	\[ \mathcal{F} = f \mod m^k \mathcal{O}[\epsilon]  \]derived from \eqref{eq:quotient}.  
	This yields that $ g \in (f, m^k)$.
	In other words, we have the surjective map
	\[
		\mathcal{V}: (f, m ^k ) \to 
		T \Def_{k}^X
	\]
	which sends $ g $ to the infinitesimal deformation $ V(f + \epsilon g)$. Notice that different $g $ may gives rise to the same deformation. To understand the kernel of $\mathcal{V}$, we shall determine the condition that when $\mathcal{X}$ is equivalent to the trivial deformation $X \times \mathcal{T}_{\epsilon}$ (represented by the zero locus of $f$).
	From the lifting lemma, the isomorphism between $ \mathcal{X}$ and $ X\times \mathcal{T}_{\epsilon}$, can be lifted to the automorphism
	$ \phi: \mathbb{C}^n\times \mathcal{T}_{\epsilon} \to\mathbb{C}^n\times \mathcal{T}_{\epsilon} $ satisfying the following conditions:
	\begin{enumerate}
		\item[(1)] $ \phi|_{ \mathbb{C}^n \times 0 }  = \id $;
		\item[(2)] $ \phi|_{P \times \mathcal{T}_\epsilon} = \id $;
		\item[(3)] applying the automorphism $\phi $, the ideal $( f )$ coincides with $ (\mathcal{F}) $ in the local ring $ \mathcal{O}[\epsilon] $.
	\end{enumerate}
	Using the condition (1), we represent the pullback morphism 
	$ \phi^* : \mathcal{O}[\epsilon] \to \mathcal{O}[\epsilon] $ 
	 by
	  \[ \phi^*(x_i) = x_i + \epsilon \delta_i (x_1,\ldots,x_n) \]
	   with $ \delta_i (x) \in \mathcal{O}$. 
	   The restriction of $ \phi$ on $ P \times \mathcal{T}_\epsilon $ is represented by  
	   \begin{align}
		\phi^* : \mathcal{O}[\epsilon]/(f, m^k ) & \to 	\mathcal{O}[\epsilon] / (f,m^k ),\\
	x_i	& \mapsto  x_i + \epsilon \delta_i(x_1, \ldots, x_n) \mod (f,m^k).
	   \end{align}
	   Now the condition (2) yields $ \phi^* = \id \mod (f,m^k)$, and therefore, $ \delta_i \in (f, m^k ) $.
	   From Taylor expansion, we have 
	\[
		f(\phi(x_1), \ldots, \phi(x_n))= f(x_1,\ldots,x_n) + \sum_{i=1}^n \delta_i (x_1,\ldots,x_n) \partial_i (f) \epsilon.
	\]
	The deformation $X$ is equivalently trivial if and only if the principle ideals 
	\[
		\left(f(x_1,\ldots,x_n) + \sum_{i=1}^n \delta_i(x_1,\ldots,x_n) \partial_i (f)\epsilon \right) \]
		and
		\[ \left(
			f(x_1,\ldots,x_n)+g(x_1,\ldots,x_n) \epsilon \right) 
	\]
	coincide.
	Since the invertible element in $\mathcal{O}[\epsilon]$ is contained in $ \mathbb{C} \oplus \epsilon \mathcal{O}  $, it follows that 
	\begin{align*}
		& f(x_1,\ldots,x_n) + \sum_{i=1}^n \delta_i(x_1,\ldots,x_n) \partial_i (f)\epsilon \\
		& = (1 + \alpha \epsilon ) (f(x_1,\ldots,x_n)+g(x_1,\ldots,x_n) \epsilon) \\
		&= f(x_1,\ldots,x_n)+(g(x_1,\ldots,x_n) + \alpha f(x_1,\ldots,x_n)) \epsilon
	\end{align*}
	for some $ \alpha \in \mathcal{O}$.
	This implies that 
	\[ g = -\alpha f + \sum_{i=1}^n \delta_i(x_1,\ldots,x_n) \partial_i (f) \in \left( f, m^k J(f) \right). \]
	Hence, the deformation represented by $ g $ is trivial if and only if $ g \in \left( f, m^k J(f) \right) $. This means that the kernel of $ \mathcal{V}$ equals $ (f, m^k J(f))$, so the isomorphism \eqref{equ:TDef} holds.
	The second assertion follows easily  from the exact sequence
\[
	0 \to T \Def_{k}^X \to A_k(f) \to  \mathcal{O}_{f,k}  \to  0.
\]
\end{proof}

% describe the connections of Mt, At, Jt, Kt
\section{Connections of Hilbert-Poincar\'{e} Series $ \mathbb{M}(t) , \mathbb{A}(t), \mathbb{J}(t) , \mathbb{K}(t)$}\label{sec:Connections}
In this section, we would like to construct Hilbert-Poincar\'{e} Series, denoted respectively by $ \mathbb{M}(t) , \mathbb{A}(t), \mathbb{J}(t) , \mathbb{K}(t)$, associated to a given hypersurface singularity $ ( V(f),0 )$ and investigate their relations.
\subsection{The Relation between $\tau_k$ and $ \mu_k $}

 Now we assume that $ f \in \mathcal{O}$ is a weighted homogeneous polynomial. Then the ideals $\left ( f, m^k J(f) \right ) $ and $   m^k J(f) $ coincide when $k = 0 $ or $1$.  As a consequence, we have 
\[
	M_0(f) = A_0(f) \qquad M_{1}(f) = A_1(f). 	
\]
In general, we consider the exact sequence
\[
	0 \to R_k \to  M_k(f) \to A_k(f) \to 0  ,
\]
where 
\[ R_k :=\frac{\left ( m^{k}J(f) , f \right ) }{ m^{k}J(f)} \cong \frac{\mathcal{O}\cdot f }{m^{k}J(f) \cap  ( \mathcal{O}\cdot f)   } =\frac{ \mathcal{O} \cdot f }{ m^{k-1}  \cdot f  } \cong \frac{\mathcal{O}}{m^{k-1}}
\]
by applying $ (f) \subseteq m J(f) $. 
Hence, for $k \geqslant 2 $ we have 
\[
	\dim R_k = \binom{k-2+n}{n}. 
\]
Therefore, we get   
\begin{equation}\label{equ:dimMA}
	\mu_k = \tau_k +  	\binom{k-2+n}{n} .
\end{equation}
Recall the series $\mathbb{A}_f(t) $ and $ \mathbb{M}_f(t) $ associated to $A_k(f)$ and $ M_k(f) $ respectively. 
Applying the equation~\eqref{equ:dimMA} we formulate the connection between $\mathbb{A}_f(t) $ and $ \mathbb{M}_f(t) $, expressed by the formula
	\begin{equation}\label{equ:MA}
		\mathbb{M}_f(t)=\mathbb{A}_f(t)+{\frac{t^2}{(1-t)^{n+1}}}.
	\end{equation}
\subsection{Graded Glgebra of $m$-Filtration of Jacobian ideal} 
Recall $ J(f) \subseteq \mathcal{O}$ the Jacobian ideal of $ f $. By multiplication with $ m^i $, we define the filtration of $ J(f) $:
\[
	J(f)   \supseteq     m J(f)    \supseteq   m^2 J(f)     \supseteq m^3 J(f)  \supseteq \cdots.
\]
Denote the relative quotient 
\[ J_k(f) =\gr_k J(f) \cong \dfrac{m^k J(f)}{  m^{k+1} J(f ) }.
\]
Then we get the exact sequence
\[
	0\to J_k(f) \to M_{k+1}(f) \to M_{k}(f) \to 0	.
\]
It follows the following useful lemma.
\begin{lem}\label{lem:dimMk}
	The $k$-th Tjurina number of $V(f)$ is given by
	\[
		\mu_k =  \mu_0  + \sum_{i=0}^{k-1} \dim J_{i} (f).
	\]
\end{lem}
For a monomial $a $ of degree $d$, the map
\[
	J_k(f) \to J_{k+ d} (f), \quad  [g]  \mapsto  [a \cdot g]
\]
is well-defined.
Now we fix the coordinate $(x_1,\ldots, x_n ) $ of $  \mathbb{C}^n $, and then the graded algebra $ \gr  \mathcal{O} := \oplus_{i = 0}^{\infty} m^i / m^{i+1}  $ is natural isomorphic to the polynomial ring
\[
	 \mathcal{P} : = \mathbb{C}[x_1, \ldots, x_n] . 
\]
\begin{defn}
	In this way, the graded vector space
\[ \gr J(f) := \oplus_{k=0}^{\infty } J_k(f),
\]
admits a graded $ \mathcal{P} $-module structure with respect to the coordinate $ \mathbf{x} = (x_1, \ldots, x_n ) $. 
\end{defn}
More explicitly, the degree of a non-vanishing element $ a \partial_i (f) \in \gr J(f) $ is defined as $ \deg (a)  $. Denote by the $\mathbb{J}_{f}(t) $ the Hilbert-Pancar\'{e} polynomial of $\gr J(f)$: 
\[
	\mathbb{J}_{f}(t) = \sum_{t=0}^{\infty} \dim (J_{k} ) \cdot  t^k  .
\]
\begin{lem} \label{lem:AMJ}
	Suppose that $ \mathbb{A}_f(t) ,\mathbb{M}_f(t) ,\mathbb{J}_f(t) $ are the series defined associated to the weighted homogeneous polynomial $ f $. We obtain the relations 
	\begin{equation}\label{equ:Mt}
		\mathbb{M}_f(t) =	\frac{ \mu_0 + \mathbb{J}_f(t) t }{1-t}
	\end{equation}
	and  
	\begin{equation}\label{equ:At}
		\mathbb{A}_f(t) =	\frac{ \mu_0 + \mathbb{J}_f(t) t }{1-t} - \frac{t^2}{(1-t)^{n+1}}.
	\end{equation}
\end{lem}
\begin{proof}
	From  Lemma~\ref{lem:dimMk}, we have 
	\begin{align*}
		\mathbb{M}_f(t)  & = \sum_{k=0}^\infty \mu_k t^k \\
		& = \sum_{k=0}^\infty \left( \dim M_0(f) t^k  + \sum_{i=0}^{k-1} J_{i} (f) t^i t^{k-i}  \right) \\
		& = \frac{1}{(1-t)} \mu_0  + \sum_{j=1}^\infty \sum_{i=0}^{\infty} J_{i} (f) t^i \cdot t^{j}\\
		& =  \frac{1}{(1-t)} \mu_0 + \sum_{j=1}^\infty   t^{j} \mathbb{J}_f(t)\\
		& = \frac{1}{(1-t)}\mu_0 +   \frac{t}{(1-t)}  \mathbb{J}_f(t).
	\end{align*}
	This confirms Equation \eqref{equ:Mt}. The equation \eqref{equ:At} is derived from Equations \eqref{equ:MA} and \eqref{equ:Mt}.
\end{proof}
\subsection{Representation of $\gr J(f)$}
Inspired by Lemma \ref{lem:AMJ}, it suffices to compute $\mathbb{J}_f(t)$. This can be achieved by constructing the free resolution of $ \gr J(f)$.
A free graded $ \mathcal{P} $-module in this paper means the free $ \mathcal{P} $-module generated by finite variables $ \nu_1, \ldots, \nu_r$ with $ \deg( \nu_i) \in \mathbb{Z} $. Such $ \mathcal{P} $-module will be denoted by
$\mathcal{P} \left \langle \nu_1, \ldots, \nu_r \right \rangle$.
The Hilbert-Poincar\'{e} series of graded module $ E = \oplus_i E_i $ is defined as 
\[
	\mathbb{E}(t) = \sum_{i= 0}^{\infty} \dim (E_{i}) t^i	. 
\]
\begin{notation}
	For an element $ g $ of graded module $E = \oplus_i E_{i}$, we denote by $\mh (g)$ the minimal nonzero homogeneous part of $g$. Explicitly, if $ g = \sum_{i\geqslant i_0} g_i \in E$, with $ g_i \in E_{i} $ and $ g_{i_0} \not = 0 $, then $
	\mh (g) = g_{i_0}$.
\end{notation}

Now we return to the polynomial $f(x_1,\ldots, x_n)$.
Let $ E^{(1)} $ be the free graded $\mathcal{P}$-module generated by $ v_1 , \ldots, v_n $ with $ \deg (v_i) = 0 $ for $i = 1, \ldots , n$. Let $ f_i $ be the partial derivative of $f$ for each $ i $.
There exists a natural degree-preserving epimorphism
\[
	f_* : E^{(1)} \to \gr J(f)
\]
which sends $ a v_i $ to $ ( \ldots ,0,  a f_i , 0, \ldots ) \in J_ {\deg(a) }(f) $.
For determining the kernel of $f_*$, we define elements of $E^{(1)}$:
\[
	\mathcal{T}_{i,j}  :=  f_i v_j - f_j v_i.
\]
Identifying $ v_i $ with partial derivative operator $ \partial_i$,
we obtain the action of $ E^{(1)} $ on $\mathcal{P}$.
A useful fact says that if a derivation $D \in E^{(1)}$ gives $D (f) = 0  $, then $D $ is generated by $\mathcal{T}_{i,j} $.

\begin{defn}\label{def:Kf}
	Let $K(f) $ be the graded $ \mathcal{P} $-submodule of $ E^{(1)} $ generated by all the homogeneous elements of the form
	\begin{equation}\label{equ:Tij}
		\mh \left( \sum_{i<j } a_{i,j} \mathcal{T}_{i,j}  \right).	
	\end{equation}
	Denote by $ \mathbb{K}(t) $ the Hilbert-Poincar\'{e} series of $K(f)$. 
\end{defn}
It can be checked directly that every homogeneous element of $ K(f) $ is hence of the form \eqref{equ:Tij}.
%  We would like to give a characterization of the kernel of $f_*$.
\begin{lem}\label{lem:Kf}
	The kernel of $f_*$ is identical to the module $K(f)$.
\end{lem}
\begin{proof}
	Suppose that $ \kappa \in K(f) $ is homogeneous of degree $ k $ such that $ f_*(\kappa) = 0 $. So $ \kappa $ is expressed as
	\[
		\kappa = \sum_{i=1}^n a_i v_i,
	\]
	where each coefficient $a_i$ is either zero or of degree $ k $. 
	Then the fact that $f_*( \kappa ) $ vanishes in $ \gr J(f)$ is equivalent to saying that
	\[
		f_* (\kappa) =\sum_{i=1}^n a_i f_i \in   m^{k+1} J(f) .\]
	Write 
	\[
		\sum_{i=1}^n a_i f_i = \sum_{i=1}^n b_i f_i
	\]
	where $ b_i \in m^{k+1} $. Set 
	\[
		D=\sum_{i=1}^n (a_i -b_i) v_i = \kappa - \sum_{i=1}^n b_i v_i .
	\]  
	We have $ D(f) = 0 $ by viewing $D$ as a derivation.  
	From the fact concerning derivations above, we see that 
	$  D $ is generated by $ \mathcal{T}_{i,j} $.
	It follows that
	\[
		\kappa-\sum_{i=1}^n b_i v_i  = \sum_{i<j} c_{i,j} \mathcal{T}_{i,j}
	\]
	for some $ c_{i,j} \in \mathcal{O} $ and thus 
	\[
		\kappa   =\mh \sum_{i<j} c_{i,j} \mathcal{T}_{i,j}.
	\]
	This implies that $K(f) \subseteq \ker(f_*)$. The converse can be deduced in the same manner.
\end{proof}
As a consequence, we have the following result.
\begin{cor}\label{cor:JtKt}
	Assume that $ f$ is weighted homogeneous. Let $ \mathbb{J}_f(t)$, $ \mathbb{K}_f(t)$ be defined above associated to the singularity $ V(f) $. 
	Then we have 
	\begin{equation}\label{equ:JtKt}
		\mathbb{J}_f(t) = \frac{n}{(1-t)^n}-\mathbb{K}_f(t) .
	\end{equation}
\end{cor}
\begin{proof}
	The Hilbert-Poincar\'{e} series of $E^{(1)}$ is equal to 
	\[
		\sum_{i = 0}^{\infty }\dim (E^{(1)})_i t^i = n \sum_{i= 0}^{\infty} \binom{i +n-1}{n-1} t^i =  \frac{n}{(1-t)^n} .
	\] 
	Now the corollary follows directly from the exact sequence
	\[
		0 \to K(f) \to 	E^{(1)} \to \gr J(f) \to 0	. \]
\end{proof}
It is trivial to see that $J_0(f)$ is a linear combination of $\{v_1, \ldots v_n\}$, and therefore
\[ \dim J_{0}(f) = n . \]
This implies that 
\[
	\mathbb{J}_f(t) = n \mod t. 	
\]
This formula can be generalized up to the multiplicity of $f $.
\begin{cor}\label{cor:Jtmod}
	 Denote by $ m_0, m_i $ the multiplicity of $f , f_i$ respectively. Then
	\[
		\mathbb{J}_f(t) =  \frac{1}{(1-t)^n} \left( n-\frac{1}{2} c ( 2 n -c- 1) \cdot t^{m_0-1}  \right) \mod t^{m_0},
	\]
	where $ c = \# \{ i \in \{ 1, 2, \ldots, n \} | m_i = m_0-1 \}$. 
\end{cor}
\begin{proof}
	It is trivial to see that $ m_i \geqslant m_0 -1 $.
	Without loss of generality, we assume that 
	\[
	m_0-1= m_1  = \ldots = m_c < m_{c+1} \leqslant m_{c+2} \leqslant \cdots \leqslant m_{n} .
	\]
	With this assumption, we have
	    \[ \mh \mathcal{T}_{i,j} = \begin{cases}
			\mh(f_i) v_j - \mh(f_j)  v_i & \text{ for $ 1\leqslant i < j \leqslant c$}; \\
			\mh(f_i) v_j & \text{ for $ 1\leqslant i  \leqslant c < j \leqslant n $}; \\
			0 \mod m^{m_0-1} E^{(1)} & \text{ for $ c+1 \leqslant i < j \leqslant n $}.
		\end{cases}
		\]
		Since each homogeneous element $ \kappa \in K(f) $ is of the form \eqref{equ:Tij},
	 the degree of $\kappa $ is not less than $m_0-1$. Moreover, the homogeneous part $K(f)_{m_0-1}$ of $K(f)$ is the $ \mathbb{C} $-linear space spanned by elements:
	    \[  \mh(f_i) v_j - \mh(f_j)  v_i  \text{ for $ 1 \leqslant i < j \leqslant c$ }\]
		and  
			\[ \mh(f_i) v_j  \text{ for $ 1\leqslant i  \leqslant c < j \leqslant n $}. \]
	These elements are linearly independent, and thus 
	\[ \dim K(f)_{m_0-1} = \frac{1}{2}c ( 2n -c- 1) . \]  
	 This yields
	 \begin{align*}
		\mathbb{K}_f(t) &= \sum_{i=0}^{\infty} \dim K(f)_i t^i \\
		& =  \dim K(f)_{m_0 -1 } t^{m_0 -  1}\\
		 & = \frac{1}{2}c ( 2n -c- 1) \cdot t^{m_0-1} \mod t^{m_0}.
	 \end{align*}
	Now the corollary follows from Corollary \ref{cor:JtKt}.
\end{proof}
One can verify that the number $ c $ in Corollary \ref{cor:Jtmod} is equal to the cardinality of weights which achieve the maximal value. As follows, the corollary \ref{cor:dimlessk} is the restatement of Main Theorem \ref{thm:muIntro}, which concludes that $k$-th Tjurina numbers (resp. Milnor numbers) with $k$ up to multiplicity are determined by the weights of variables. 
\begin{cor}\label{cor:dimlessk}
	Let $f$ be an isolated weighted homogeneous singular of type $(w_1, \ldots, w_n)$, which satisfies $w_1=w_2 = \ldots = w_c$ and $ w_i < w_c $ for $i > c $. For $ k \leqslant m_0 $, the $k$-th Milnor numbers and $ k $-th Tjurina numbers are given by formulas \eqref{eq:muk}\eqref{eq:mum0}\eqref{eq:tauk}\eqref{eq:taum0}.
\end{cor}
\begin{proof}
	From Lemma \ref{lem:AMJ} and Corollary \ref{cor:Jtmod}, we have  
	\begin{align*}
		\mathbb{M}_f(t) & =	\frac{ \mu_0  }{1-t} + \frac{t}{(1-t)^{n+1}} \left( n-\frac{1}{2} c ( 2 n -c- 1) \cdot t^{m_0-1}  \right) \mod t^{m_0+1} \\
		& =  \mu_0  + \frac{t}{(1-t)^{n+1}} \left( n-\frac{1}{2} c ( 2 n -c- 1) \cdot t^{m_0-1}  \right) \\
		& = -\frac{1}{2} c ( 2 n -c- 1) t^{m_0}+ \sum_{i=0}^{m_0} \left( \mu_0 + n \binom{i-1+n}{n} \right) t^i 
	\end{align*}
	and  
	\begin{align*} 
		\mathbb{A}_f(t) &=	\frac{ \mu_0  }{1-t} + \frac{t}{(1-t)^{n+1}} \left( n-t-\frac{1}{2} c ( 2 n -c- 1) \cdot t^{m_0-1}  \right) \mod t^{m_0+1} \\
		&= -\frac{1}{2} c ( 2 n -c- 1) t^{m_0}+ \sum_{i=0}^{m_0} \left( \mu_0 + n \binom{i-1+n}{n} -\binom{i-2+n}{n}  \right) t^i . 
	\end{align*}
	Then the formulas are derived from the expression of series above. 
\end{proof}

\section{Representation of $ K (f) $}\label{Representation}
\subsection{New Koszul Type Complex}
From the last section, we know that both $ \mathbb{A}_f(t) $ and $ \mathbb{M}_f(t) $ depend on the kernel $ K(f) $. So our computation reduces to resolving the $\mathcal{P}$-module $ K(f) $. 
Define $ B^{(1)}$ to be the free $\mathcal{P}$-module generated by $\mh \mathcal{T}_{i,j}$, i.e., $ B^{(1)} = \mathcal{P} \left \langle \mh \mathcal{T}_{i,j} \right \rangle $, where $ i,j $ range over $1 \leqslant i < j \leqslant n$. From the definition, $B^{(1)}$ is a submodule $ K (f) $. For some cases (e.g. the case of Brieskorn-Pham singularity \cite{brieskorn1966beispiele,  pham1965formules}), $  B^{(1)} $ coincides with $K(f) $, but generally this is not true. To manipulate the differences between $B^{(1)}$ and $K(f)$, we make use the Koszul complex and the modified versions. 
\begin{defn}
	For $ r \geqslant 1 $, let $E^{(r)}$ be the free ${P}$-module generated by $v_{i_1,\ldots, i_r }$  with $1 \leqslant i_1< i_2 <\cdots  <i_r \leqslant n $.
	For fixed indexes $i_1, \ldots, i_{r+1} $, we define
	\[
	\mathcal{T}_{i_1,\ldots,i_{r+1}} := 	\sum_{j=1}^{r+1} (-1)^{j+1}  f_{i_j} v_{i_1, \ldots, \hat {i_j} , \ldots i_{r+1}}  \in E^{(r)},
	\]
	where $ f_i $'s denote the partial derivatives of $f $.
	The classical Koszul complex $\Kos^1_*(f)$ is the sequence of $ P $-modules 
\[
\begin{tikzcd}
	0 \ar[r] & E^{(n)}\ar[r, "d_{n}"]  & E^{(n-1)}\ar[r] & \cdots  \ar[r,"d_2"] &  E^{(1)} \ar[r,"d_1"] &  J(f)\ar[r] & 0,
\end{tikzcd}
\]
where $ d_1 (v_i) = f_i $ and the homomorphism
$
d_{r+1} : E^{(r+1)} \to E^{(r)} $
sends $v_{i_1,\ldots, i_{r+1}}$ to $\mathcal{T}_{i_1,\ldots, i_{r+1}}$.
\end{defn}
Since $ V(f) $ is an isolated weighted homogeneous singularity, the sequence of partial derivatives $ f_1, \ldots , f_n $ is a regular sequence. It is well known that Koszul complex $\Kos^1_*(f)$ is exact and so is its formal completion $\Kos^1_*(f) \otimes_{\mathcal{P}} \mathcal{O}$ at the origin.

\begin{notation}
	Denote by $ m_i $ the multiplicity of $f_i$.
	For $1\leqslant i_1 < i_2 < \ldots, i_r \leqslant  n $ and $ r \geqslant 2 $, we define
	\[
		m_{i_1,\ldots,i_{r}} := m_{i_1} + \cdots + m_{i_r} -\max \{ m_{i_1} , \ldots , m_{i_r} \}	.
	\]
	Extensively, we endow $ E^{(r)} $ with a graded structure by setting $\deg (v_i ) = 0$ and 
	\[ \deg( v_{i_1, \ldots,  i_{r}} ) = m_{i_1, \ldots,  i_{r}} \text{ for $ r \geqslant 2 $}. 
	\]
\end{notation}

% Using this notation, we obtain 
% \[
% 	\mh \mathcal{T}_{i,j}= \xi_{i}^{j} \mh (f_i) v_j - \xi_{j}^{i}\mh(f_j) v_i , 
% \]
% viewed as an homogeneous element of degree $m_{i,j}$ contained in $ E^{(1)}$.
\begin{lem}
	For $ r \geqslant 1$, the minimal homogeneous part of $\mathcal{T}_{i_1,\ldots,i_{r+1}} \in E^{(r)}$ is of degree $m_{i_1,\ldots,i_{r+1}} $.
\end{lem}
\begin{notation}
	Fix index subset $I \subseteq \{ 1,\ldots,n\}$, and define $I^* \subseteq I $ in the following way.
	If there are at least two subscripts $ j,k \in I $ such that both $ m_{j} $ and $ m_{k} $ achieve the maximal value $ m_{\max} := \max \{ m_{i} | i \in I \} $, we say that $ I $ refers to no maximum and set $I^*= I$.
	 Otherwise, $ m_{i} = m_{\max} > m_{j} $ for every $j \in I\setminus \{ i \}$, and we say that $ i $ refers to the maximum multiplicity $ m_{i }$ and set $ I^*= I \setminus \{ i \}$. 
\end{notation}
\begin{proof}
With this notation, it can be checked that $\mh\mathcal{T}_{i_1,\ldots,i_{r}}$ with $ r \geqslant 2$ can be written as 
\begin{equation}\label{eq:mhT}
	\mh\mathcal{T}_{i_1,\ldots,i_{r}}= \sum_{ i_j \in \{i_1,\ldots,i_{r}\}^* }  (-1)^{j+1}  \mh(f_{i_j}) v_{i_1, \ldots, \hat {i_j} , \ldots i_{r}} .
\end{equation}
The lemma follows by observing that 
\[ 
	\deg ( \mh(f_{i_j}) v_{i_1, \ldots, \hat {i_j} , \ldots i_{r}} ) = m_{i_j} + m_{i_1, \ldots, \hat {i_j} , \ldots i_{r}}= m_{i_1, \ldots, i_r }
\]
for each $  i_j \in \{i_1,\ldots,i_{r}\}^* $. 
\end{proof}
Now we are able to introduce the new Koszul type complex.
\begin{defn}
	Let $ \gr J(f) $ be the graded $ \mathcal{P} $-module with respect to the coordinates $ x_1, \ldots, x_n $ of $ V(f)$. 
	The new Koszul type complex $ \Kos^0_{*}(f)$ is the complex of $ \mathcal{P}$-modules defined as 
	\[
\begin{tikzcd}
	0 \ar[r] & E^{(n)}\ar[r, "\delta_{n}"]  & E^{(n-1)}\ar[r] & \cdots  \ar[r,"\delta_2"] &  E^{(1)} \ar[r,"\delta_1"] & \gr J(f)\ar[r] & 0.  
\end{tikzcd}
\]
where $ \delta_1 = f_*$ and $
\delta_{r+1} : E^{(r+1)} \to E^{(r)} $ sends
	$v_{i_1,\ldots, i_{r+1}}$ to $ \mh\mathcal{T}_{i_1,\ldots, i_{r+1}}$. 
\end{defn}
There are at least two significant differences between the classical Koszul complex and the modified version. Firstly, the new Koszul type complex $ \Kos^0_{*}(f)$ depends on the coordinates $x_1, \ldots, x_n$. Secondly, each arrow $ \delta_i $ within is viewed as a homomorphism of graded modules.  
\begin{defn}
Set
\[
	Z^{(i)}:=\ker \delta_{i}, \text{ and } B^{(i)}:= \im \delta_{i+1} . 
\]
The homology of $ \Kos^0_{*}(f)$ with $i \geqslant 0 $ is defined as 
\[
H_i (\Kos^0_* (f) ) = Z^{(i)}/B^{(i)}. \]
\end{defn}
By definition, it is trivial to see that  
\begin{equation}\label{equ:Hn}
	H_0 (\Kos^0_* (f) )= H_n (\Kos^0_*(f) ) = 0 .
\end{equation}
From lemma \ref{lem:Kf}, we have  $
	Z^{(1)} = K(f) 	$.
Therefore, 
\[
	H_{1}	(\Kos^0_*(f) ) = K(f)/ \left \langle \mh \mathcal{T}_{i,j} \right \rangle_{i<j}. 
\]
\begin{lem}\label{lem:ann}
	If the multiplicity $ m_k $ is minimal in the set $ \{ m_1, \ldots ,m_n\} $, then $H_{r}(\Kos^0_* (f) )$ is annihilated by $\mh(f_k )$ for $ r \geqslant  1$.
\end{lem}
\begin{proof}
	Without loss of generality, we set $ k = 1$.
	Given indexes $i_1,\ldots, i_r>1$, we see from the expression \eqref{eq:mhT} that
	\[
		\mh\mathcal{T}_{1, i_1,\ldots,i_{r}}=  -  \mh(f_{i_j}) v_{i_1, \ldots, \hat {i_j} , \ldots i_{r}}  + \sum_{ i_j \in \{ i_1,\ldots,i_{r}\}^* }  (-1)^{j}  \mh(f_{i_j}) v_{i_1, \ldots, \hat {i_j} , \ldots i_{r}} .
	\] 
	It follows that
	\[ \mh(f_1) v_{i_1,\ldots,i_r} =  \sum_{ i_j \in \{ i_1,\ldots,i_r \}^*  }  (-1)^{j}  \mh(f_{i_j}) v_{1,i_1, \ldots, \hat {i_j} , \ldots i_{r}} \mod B^{(r)}.
	\]
	This implies that for $ \theta \in H_{r}(\Kos^0_*(f) ) $, the product
	$ \mh(f_1) \cdot \theta $ is spanned by the set $ \{ v_{1,j_1,\cdots,j_{r-1} } \} $ with $2\leqslant j_1,\cdots,j_{r-1} \leqslant n$. That is 
	\begin{equation}\label{eq:mh}
		\mh(f_1) \cdot \theta = \sum_{j_1,\ldots, j_{r-1}} a_{j_1,\ldots, j_{r-1}} v_{1,j_1,\cdots,j_{r-1} } \mod B^{(r)}
	\end{equation}
	for some homogeneous coefficients $a_{j_1,\ldots, j_{r-1}}$.
	Therefore, 
\begin{align*}
	0= &\delta_{r} (\mh(f_1) \cdot \theta )\\
	= & \sum_{j_1,\ldots, j_{r-1}} a_{j_1,\ldots, j_{r-1}} \mh \mathcal{T}_{1,j_1,\cdots,j_{r-1} } 
	\\
	= &  \sum_{j_1,\ldots, j_{r-1}} a_{j_1,\ldots, j_{r-1}}  \mh(f_1)v_{j_1,\cdots,j_{r-1} }  \\
	& + \sum_{j_1,\ldots, j_{r-1}} \sum_{ j_k \in \{ j_1,\ldots, j_{r-1} \}^* } (-1)^{k} a_{j_1,\ldots, j_{r-1}}  \mh(f_{j_k}) v_{1,j_1, \ldots, \hat {j_k} , \ldots j_{r-1}} .
\end{align*}
	This implies that 
	$a_{j_1,\ldots, j_{r-1}} = 0 $ for subscripts $2 \leqslant j_1,\ldots, j_{r-1} \leqslant n $, and therefore $\mh(f_1) \cdot \theta  = 0 \mod B^{(r)}$. In other words, $\mh(f_1) \cdot \theta \in B^{(r)}$. 
\end{proof}

\begin{notation}\label{not:Theta}
	Let $ q $ denote the maximal factor of $ \mh \mathcal{T}_{1,2,\cdots,n}$. In other words, the polynomial $ q $ is the maximal common factor of every $ \mh (f_i)$ with $i \in \{ 1,\ldots,n\}^* $.
    We define $ \Theta_q := \frac{1}{q} \cdot \mh \mathcal{T}_{1,2,\cdots,n} $ to be a homogeneous element of $E^{(n-1)}$.
\end{notation}
\begin{thm}\label{thm:Theta}
	With the notation \ref{not:Theta}, the submodule $ Z^{(n-1)} $ is generated by the single element $\Theta_q$. Therefore we obtain the isomorphism 
	\[ H_{n-1} (\Kos^0_{*}(f) ) = \Theta_q\cdot \mathcal{P}  /  (q \Theta_q  \cdot \mathcal{P}) \cong \mathcal{P} / q . \]
	As a consequence, $ H_{n-1} (\Kos^0_{*}(f) )$ vanishes if and only if $ q $ is a constant.
\end{thm}
\begin{proof}
	By assumption, we have the decomposition 
	\[  \mh (f_i ) = q \cdot r_i  \text{ for $ i \in \{ 1,\ldots, n\}^* $}, \]
	 where such $ r_i $'s have no common factor.
	 Assume $ m_k $ is minimal. 
	Given $ \theta \in Z^{(n-1)}$, since $B^{(n-1)}$ is generated by $\mh \mathcal{T}_{1,2,\cdots,n}$, Lemma \ref{lem:ann} implies that  
		\begin{equation}\label{equ:mhfk}
			\mh(f_k) \theta = p \mh \mathcal{T}_{1,2,\cdots,n} 
		\end{equation}
	for some polynomial $ p $.
	Thus 
	\begin{align*}
		 \theta =& \frac{p}{\mh (f_k )} \mh \mathcal{T}_{1,2,\cdots,n} \\
		    =&  \sum_{i \in I_f} (-1)^{i+1} \frac{ \mh(f_i) p }{\mh (f_k )}  v_{1,2,\ldots , \hat{i} , \ldots, n }  \\
	   =&  \sum_{i \in I_f} (-1)^{i+1} \frac{ r_i p }{ r_k } v_{1,2,\ldots , \hat{i} , \ldots, n } \in E^{(n-1)}
. 
\end{align*}
We see that every $r_i p $ is divisible by $ { r_k }$ and then $p $ is divisible by  $ r_k $. Equation \eqref{equ:mhfk} yields 
\[
	\theta = \frac{p}{r_k} \Theta_q.
\]
This shows that $Z^{(n-1)}$ is generated by $ \Theta_q  $. 
\end{proof} 
We end up this section with an explicit formula for the Hilbert-Poincar\'{e} series of $\gr J(f)$ on the assumption that the modified Koszul complex is exact.
\begin{notation}
	Given $ (a_1, \ldots, a_n) $ a sequence of positive integers, we
	let $b_1 \leqslant \cdots \leqslant b_{n}$ be its reordering. Define the series  
	\[
	\mathbb{S}_{(a_1,\ldots,a_n)} (t) = \frac{1}{(1-t)^n} \left( n + \sum_{r=1}^{n-1} (-1)^r \sum_{1 \leqslant i_1 < \cdots < i_r < n} (n-i_r) t^{b_{i_1}+ \cdots + b_{i_r}} \right).
	\]	
\end{notation}
Note that the series $\mathbb{S}_{(a_1,\ldots,a_n)} (t)$ does not depend on the maximum of $a_i$'s.
\begin{thm}\label{thm:St}
	If the modified Koszul complex $ \Kos^0_{*}(f) $ is exact, then the Hilbert-Poincar\'{e} series of $\gr J(f)$ equals $ \mathbb{S}_{m_1,\ldots,m_n} (t) $. 
\end{thm}
\begin{proof}
	Denote by $\mathbb{E}^{(r)}(t)$ the Hilbert-Poincar\'{e} series of $ E^{(r)} $.  Then 
	\[
		\mathbb{E}^{(1)}(t) = \frac{n}{(1-t)^n}	
	\]
	and for $ r\geqslant 1$,
	\[
	\mathbb{E}^{(r+1)}(t)= \frac{1}{(1-t)^n} \left(  \sum_{1 \leqslant i_1 < \cdots < i_r < n} (n-i_r) t^{b_{i_1}+ \cdots + b_{i_r}} \right) .
	\]
	Then the theorem follows by the alternative sum of the Hilbert-Poincar\'{e} series of the complex $ \Kos^0_{*}(f) $.
\end{proof}

\subsection{Filtration of $K^{(r)}$}
We give a generalization of $ K(f) $ defined in Definition \ref{def:Kf}.
\begin{defn}
	Let $ K^{(r)} $ be the submodule of the free $ \mathcal{P}$-module $ E^{(r)} $ generated by all the homogeneous elements of the form 
	\[  \mh \sum_{i_1, \ldots, i_{r+1} } a_{i_1, \ldots, i_{r+1} } \mathcal{T}_{i_1, \ldots, i_{r+1} }  
	\]	
	where $a_{i_1, \ldots, i_{r+1} } $ runs over $ \mathcal{P}$.
\end{defn}
 Note that in this notation $ K^{(1)} $ coincides with $K(f) $.
\begin{lem}\label{lem:BKZ}
	Recall $B^{(r)}$ and $Z^{(r)}$ associated to modified Koszul complex $\Kos^{*}_0 (f) $. We obtain 
	\[
		B^{(r)} \subseteq K^{(r)} \subseteq Z^{(r)}.
	\]
\end{lem}
\begin{proof}
	The inclusion $B^{(r)} \subseteq K^{(r)}$ follows easily by definition.
	It suffices to show that each homogeneous element of $K^{(r)}$ must belong to $ Z^{(r)}$.
	Given a homogeneous element $ \kappa \in K^{(r)} $, there exist coefficients $a_{i_1, \ldots, i_{r+1} } \in \mathcal{P} $ such that 
	\[
		 \kappa = \mh \sum_{i_1, \ldots, i_{r+1}} a_{i_1, \ldots, i_{r+1} }  \mathcal{T}_{i_1, \ldots, i_{r+1} } .
	\] 
	We assume that  
	\begin{equation}\label{equ:sumT}
		\sum_{i_1, \ldots, i_{r+1}} a_{i_1, \ldots, i_{r+1} }  \mathcal{T}_{i_1, \ldots, i_{r+1} } =  \sum_{j\geqslant 0 } \sum_{i_1, \ldots, i_{r}} b_{i_1, \ldots, i_{r} }^{(j)} v_{i_1, \ldots, i_{r} } ,
	\end{equation}
	where each homogeneous coefficient $ b_{i_1, \ldots, i_{r} }^{(j)} \in \mathcal{P} $ either vanishes or satisfies
\[
	\deg b_{i_1, \ldots, i_{r} }^{(j)}  + m_{i_1, \ldots, i_{r}} = \deg(\kappa) + j.
\]
Using this notation, we have 
	\[ 
		\kappa = \mh \sum_{i_1, \ldots, i_{r+1}} a_{i_1, \ldots, i_{r+1} }  \mathcal{T}_{i_1, \ldots, i_{r+1} } =   \sum_{i_1, \ldots, i_{r}} b_{i_1, \ldots, i_{r} }^{(0)} v_{i_1, \ldots, i_{r} } .
	\]
Applying $d _r $ to both sides of \eqref{equ:sumT}, we get
\[ 
	0 = d_{r} \left( \sum_{i_1, \ldots, i_{r+1}} a_{i_1, \ldots, i_{r+1} }  \mathcal{T}_{i_1, \ldots, i_{r+1} } \right)  =   \sum_{j\geqslant 0 } \sum_{i_1, \ldots, i_{r}} b_{i_1, \ldots, i_{r} }^{(j)} \mathcal{T}_{i_1, \ldots, i_{r} } .
\]
Since $b_{i_1, \ldots, i_{r} }^{(0)} \mh \mathcal{T}_{i_1, \ldots, i_{r} }$  is homogeneous of degree $\deg(\kappa) $, we have 
\[
\sum_{i_1, \ldots, i_{r}} b_{i_1, \ldots, i_{r} }^{(0)} \mh \mathcal{T}_{i_1, \ldots, i_{r} } = 0  . \]
This implies that 
 \[ 
	 \delta_{r}(\kappa)=  \delta_{r} \sum_{i_1, \ldots, i_r} b_{i_1, \ldots, i_{r} }^{(0)} v_{i_1, \ldots, i_{r} }  = \sum_{i_1, \ldots, i_r} b_{i_1, \ldots, i_{r} }^{(0)} \mh \mathcal{T}_{i_1, \ldots, i_{r} }  =  0.  \]
In other words, 
$
	\kappa  \in Z^{(r)}
$.
This yields the inclusion $K^{(r)} \subseteq Z^{(r)}$.
\end{proof}
To describe all generators of $ K^{(r)} $,
we construct a natural filtration of $ K^{(r)} $.
\begin{notation}
	For a homogeneous element $ \kappa  \in K^{(r)}$, we define the level of $ \kappa $ as 
	\[
		\lev (\kappa) = \min_{ \{ a_{i_1,\ldots,i_{r+1}} \} } \left \{ \deg(\kappa) - \min_{\# I = r+1} \{ m_{I}+\mult (a_{I}) \}  \right \}	,
	\]
	where $ \{ a_{i_1, \ldots, i_{r+1}} \}$ with $ 1< i_1 < i_2 \cdots < i_{r+1} < n $ are coefficients verifying
	\[
	\kappa = \mh \sum_{\#I = r+1}a_{I} \mathcal{T}_{I}. \]
	Denote by $ K^{(r)}_l  $  the submodule of $ K^{(r)} $ generated by all the homogeneous elements of level $ \leqslant l  $.  
\end{notation}
Obviously, $ K^{(r)}_i \subseteq K^{(r)}_j $ for $i < j$. The level structure behaves well with the $\mathcal{P}$-module structure.
For homogeneous elements $ \kappa_1 $ and $ \kappa_2 $ of levels $ \leqslant l  $ and $ a, b \in \mathcal{P} $, we have 
\[
	\lev ( a \kappa_1  + b \kappa_2 ) \leqslant l. 
\]
This implies that the level of each homogeneous element of $ K^{(r)}_ l $ is less than or equal to $ l $. 
\begin{lem}
	There exists some integer $ L $, such that $ K^{(r)}_L = K^{(r)}$.
\end{lem} 
\begin{proof}
	Since $ E^{(r)} $ is a Noetherian module, the submodule $K^{(r)}$ is finitely generated.
	Let $ L $ be the maximal level of the generators of $ K^{(r)}$. By definition, the equality $ K^{(r)}_L = K^{(r)}$ automatically holds.
\end{proof}
Note that for each index subset $ I $, the homogeneous element $\mh \mathcal{T}_{I}$ is of level zero.  
It can be checked that the converse is also true, so the following lemma holds.
\begin{lem}\label{lem:levelzero}
	The homogeneous element of level zero is generated by the elements 
	$\mh \mathcal{T}_{i_1,\ldots, i_{r+1}}$ with ${i_1,\ldots, i_{r+1}}$. 
	Therefore, $ K_0^{(r)} = B^{(r)}$.
\end{lem}

\begin{cor}\label{cor:zero}
	If $ m_i $ is minimal in the set $\{m_1, \ldots, m_n \}$, then $ \mh(f_i) \kappa $ is of level zero for any $ \kappa \in K^{(r)} $.
\end{cor}
\begin{proof}
	From Lemma \ref{lem:BKZ}, we know $ \kappa \in K^{(r)}  \subseteq Z^{(r)}$. 
	In the same manner of Lemma \ref{lem:ann}, $ \mh(f_i) \kappa $ is generated by $ \mh \mathcal{T}_{i_1, \ldots, i_{r+1}} $ with $ 1 \leqslant i_1\leqslant \ldots\leqslant i_{r+1} \leqslant n $. Now the corollary follows from Lemma \ref{lem:levelzero}.  
\end{proof}
\begin{notation}
	For $ \kappa \in K^{(r)} $ or $ K^{(r)} / K^{(r)}_l  $, we denote by  
	\[
	\Ann(\kappa; K^{(r)}_l) = \left\{ p \in \mathcal{P} ~|~ p \cdot \kappa \in  K^{(r)}_l \right\}  \subseteq \mathcal{P}
	\]
	the annihilator ideal of $ \kappa $.
\end{notation}
	The Corollary \ref{cor:zero} implies that $\mh(f_i ) \in \Ann(\kappa; K^{(r)}_0)$ whenever $ m_i $ is minimal. 
\begin{defn}
	We say that $ l \geqslant 1 $ is a gap number of $ K^{(r)} $ if $ K_{l}^{(r)} \not= K_{l-1}^{(r)} $.  
\end{defn}
Assume that $ L_1, \ldots, L_k $ are all the gap numbers of $ K^{(r)} $.
We obtain the filtration of $ K^{(r)}$:
\[
	B^{(r)} = K^{(r)}_0  \subsetneq K_{L_1}^{(r)} 	\subsetneq K_{L_2}^{(r)}  \subseteq \cdots \subsetneq K_{L_r}^{(r)}   = K^{(r)}
\] 
and for $ i \geqslant 0 $,
\[
	K_{L_{i}}^{(r)} = \cdots =  K_{L_{i+1} -1}^{(r)} 
\]
by setting $ L_0 = 0 $.
In particular, $ K(f) $ contains no gap numbers if and only if 
$ K(f) = B^{(1)} $.  
\subsection{Filtration of $ Z^{(r)} $}
This section is devoted to giving a filtration of $ \mathcal{P} $-module $ Z^{(r)}$. 
Let $\mathcal{P}[\mu] = \mathbb{C} [ x_1 , \ldots , x_n , \mu] $ be a polynomial ring.  We view $ \mathcal{P}[\mu] $ as a graded algebra over $ \mathcal{P} $, in which graded structure of $ \mathcal{P}[\mu]$ is determined by setting $\deg(x_i) = 1$ and $ \deg (\mu) = -1$. 

For $ r \geqslant 1 $, define the free graded $ \mathcal{P}[\mu] $-module:
\begin{align*}
	E^{(r)}[\mu] = E^{(r)}\otimes_{\mathcal{P}} \mathcal{P}[\mu] = \mathcal{P}[\mu] \left \langle v_{i_1,\ldots,i_r}  \right \rangle_{i_1,\ldots,i_r}
\end{align*}
with $\deg(v_i) =0 $ and $\deg(v_{i_1, \ldots, i_r} ) = m_{i_1, \ldots, i_r} $.
\begin{notation} For $ r \geqslant 2 $, we define the homogeneous element of $ E^{(r)}[\mu] $:
	\[  \mathcal{T}^{\mu}_{i_1,\ldots,i_{r+1}} =  \sum_{j=1}^{r+1} (-1)^{j+1}  f_{i_j}(\mu \mathbf{x}) \cdot \mu^{m_{i_1, \ldots, \hat {i_j} , \ldots i_{r+1}}-m_{i_1,\ldots,i_{r+1}}} \cdot v_{i_1, \ldots, \hat {i_j} , \ldots i_{r+1}}    ;
	\] 
	and 
	\[
		\mathcal{T}^{\mu}_{i,j} = \frac{1}{\mu^{m_{i,j}}} \left( f_i (\mu\mathbf{x} ) v_i - f_j(\mu\mathbf{x} ) v_i \right).
	\]
\end{notation}
Note that
\[ m_{i_j} + m_{i_1, \ldots, \hat {i_j} , \ldots i_{r+1}}-m_{i_1,\ldots,i_{r+1}} \geqslant  0. \]
The equality holds if and only if $m_{i_j}$ does not achieve the unique maximum value.
It follows that 
\[
	\mathcal{T}^{\mu}_{i_1,\ldots,i_{r+1}} = 	\mh \mathcal{T}_{i_1,\ldots,i_{r+1}} \mod \mu E^{(r)}[\mu].
\]

\begin{defn}[Extended Koszul complex]
Define the extended Koszul complex $ \Kos_*^{\mu}(f)$ as a complex of graded $\mathcal{P}[\mu]$-modules:
\[
\begin{tikzcd}
	0 \ar[r] & E^{(n)}[\mu] \ar[r, "\Delta_{n}"]  & E^{(n-1)}[\mu] \ar[r] & \cdots  \ar[r,"\Delta_2"] &  E^{(1)}[\mu] \ar[r,"\Delta_1"] &  \mathcal{P}[\mu] \ar[r]  & 0    
\end{tikzcd}
\]
where $ \Delta_i $'s are the degree-preserving morphisms defined as
 $ \Delta_1 (v_i) = f_i (\mu \mathbf{x}) $ and 
$
\Delta_{r+1} (v_{i_1,\ldots, i_{r+1}})= \mathcal{T}^{\mu}_{i_1,\ldots, i_{r+1}}$ for $ r \geqslant 0 $.
\end{defn} 
Roughly speaking, the classical Koszul complex corresponds to $ \mu = 1 $, while the modified complex corresponds to $ \mu = 0 $ (except for $ \Delta_1 $).
\begin{notation}
For $l \geqslant 1$, define the quotient
\[
	\mathcal{P}[\mu]_l:= \mathcal{P}[\mu]/\left\langle \mu^{l} \right \rangle \]
and
\[
	E^{(r)}[\mu]_{l} := 	E^{(r)}[\mu] / \left\langle \mu^{l} \right \rangle  \cong  E^{(r)} \otimes_{\mathcal{P}} \mathcal{P}[\mu]_l.
\] 
Automatically, we have $\mathcal{P}[\mu]_1 \cong \mathcal{P} $ and $E^{(r)}[\mu]_{1} \cong  E^{(r)}$.
For $ l < k $, we define the $\mathcal{P}$-morphisms 
\[ \Uparrow_k : \mathcal{P}[\mu]_l \to \mathcal{P}[\mu]_k \]
and 
\[ \Uparrow_k:
E^{(r)}[\mu]_l \to E^{(r)}[\mu]_{k} 
\]
by multiplying $\mu^{k-l}$.
Define the $l$-truncation maps 
\[ \Downarrow_l : \mathcal{P}[\mu] \; ( \text{resp. } \mathcal{P}[\mu]_k) \to \mathcal{P}[\mu]_l \]
and 
\[ \Downarrow_l:
E^{(r)}[\mu]\; ( \text{resp. } E^{(r)}[\mu]_k)  \to E^{(r)}[\mu]_{l} \]
by cutting off higher order terms.
\end{notation}

The extended Koszul complex $ \Kos_*^{\mu}(f)$ tensoring with $\mathcal{P}[\mu]_l$ becomes 
\[
\begin{tikzcd}
	0 \ar[r] & E^{(n)}[\mu]_l \ar[r, "\Delta_{n}\otimes \Downarrow_l"]  & E^{(n-1)}[\mu]_l \ar[r] & \cdots  \ar[r,"\Delta_2\otimes \Downarrow_l"] &  E^{(1)}[\mu]_l \ar[r,"\Delta_1\otimes \Downarrow_l"] &  \mathcal{P}[\mu]_l  .  
\end{tikzcd}
\]
\begin{notation}
For $ r\geqslant 1 $ and $ l \geqslant 1 $,  we set 
\[
	\tilde{Z}_{l}^{(r)}  = \ker  (\Delta_r\otimes \Downarrow_l )\subseteq E^{(r)}[\mu]_l;
\]
and let ${Z}_{l}^{(r)}$ be the leading coefficient of $ \tilde{Z}_{l}^{(r)}$, i.e., 
\[
	 {Z}_{l}^{(r)}  =  \Downarrow_1(\tilde{Z}_{l}^{(r)} )  \subseteq E^{(r)} .
\]
For completeness, we also let $ \tilde{Z}_{0}^{(r)} = 0  $.
Define the limit of sequence $ Z_{l}^{(r)}, l = 1,2,3,\ldots, $ as 
\[
	Z_{\infty}^{(r)} := 	\cap_{l} Z^{(r)}_{l}.
\]
\end{notation}
Since the kernel of $ \Downarrow_1\; : \tilde{Z}^{(r)}_{l}  \to {Z}^{(r)}_{l}$ equals $ \Uparrow_l \tilde{Z}^{(r)}_{l-1}$, we obtain isomorphisms
\begin{equation}\label{equ:Zr}
	\Downarrow_1\;: \tilde{Z}^{(r)}_{l}  / \Uparrow_l  \tilde{Z}^{(r)}_{l-1} \xrightarrow{\text{iso}} {Z}^{(r)}_{l} 
\end{equation}
and 
\begin{equation}\label{equ:rho1}
	 \Downarrow_1\;: \tilde{Z}^{(r)}_{l} / \left(  \Uparrow_l  \tilde{Z}^{(r)}_{l-1} + \Downarrow_l \tilde{Z}^{(r)}_{l+1} \right) \xrightarrow{\text{iso}}   {Z}^{(r)}_{l} /\Downarrow_l {Z}^{(r)}_{l+1}
\end{equation}
by abuse of notation.
In particular, by identifying $E^{(r)}[\mu]_1$ and $E^{(r)}$ we get 
\[
	\tilde{Z}_{1}^{(r)}=Z_{1}^{(r)} = \ker \delta_r = Z^{(r)} . 
\]
Obviously, $ Z_{i}^{(r)} \supseteq Z_{j}^{(r)} $ for $ i < j $.
Now we give a characterization of $Z_{\infty}^{(r)}$.
\begin{thm}\label{thm:Zinf}
With the notations above, we get  
\[
	Z_{\infty}^{(r)} = K^{(r)}	.
\]
\end{thm}
\begin{proof}
	According to Lemma \ref{lem:BKZ}, we have already shown that $ K^{(r)} \subseteq Z^{(r)} = Z^{(r)}_1 $. In the same manner, for any $ l \geqslant 2 $, the inclusion $K^{(r)} \subseteq Z^{(r)}_l $ also holds. Thus
	\[
	\cap_{l =1}^{\infty} Z^{(r)}_{l}  \supseteq K^{(r)}	.
\]
	Let $ \theta \in \cap_{l} Z^{(r)}_{l} $.
	Then we can find the sequence $ 
		\theta_l^\mu \in \tilde{Z}^{(r)}_l 	
	$ with $ l \in \mathbb{Z}_{\geqslant 1} $ such that $\Downarrow_1 (\theta_l^\mu ) = \theta$ and 
	$\Downarrow_{l-1} (\theta_{l}^\mu) =\theta_{l-1}^\mu $.
	Define 
	\[ \theta_{\infty}^\mu:= \varprojlim_{l \to \infty} \theta_l^\mu \in \mathcal{O}[[\mu]] \otimes_{\mathcal{P}} E^{(r)}.
	\]
	By definition, we have
	\[  (\Delta_r \otimes \Downarrow_l ) (\theta_l^{\mu}) = 0 \mod \mu^l \]
	and then  
	\[ \Delta_r (\theta_{\infty}^{\mu})  =  \lim_{l \to  \infty} (\Delta_r \otimes \Downarrow_l) (\theta_l^{\mu}) = 0 . \]
	Substituting $\mu = 1 $ into the expression of $ \theta^{\mu}_{\infty} $, the resulting element $\theta_\infty^{1} \in \mathcal{O}\otimes_{\mathcal{P}} E^{(r)}$ satisfies 
	\[
		d_r (\theta_\infty^{1}) = 0 \text{ and } \mh \theta_\infty^{1} = \theta . \]
	Since the classical Koszul complex is exact over the local ring $ \mathcal{O}$,  there exists some element  $\Theta \in \mathcal{O} \otimes_{\mathcal{P}} E^{(r+1)}$ verifying 
		$ \theta_\infty^{1} = d_{r+1} \Theta $. 
		Write
	\[	d_{r+1} \Theta= \sum_{I} \alpha_I \mathcal{T}_I + \beta_I \mathcal{T}_I  \]
	where $ \alpha_I \in \mathcal{P}$, $\sum_{I} \alpha_I \mathcal{T}_I \in E^{(r+1)}$, and $ \beta_I \in \mathcal{O}$ satisfying $\mult (\beta_I) + m_I > \deg(\theta)$.
	Then
	\[
		\theta = \mh \theta_\infty^{1} = \mh  (\sum_{I} \alpha_I \mathcal{T}_I ) \in K^{(r)}.
	\]
	This implies
	$
		\cap_{l} Z^{(r)}_{l}  \subseteq K^{(r)}	
	$.
\end{proof}
\begin{notation}
	We call $ l $ ($\geqslant 1 $) a gap number of $ Z^{(r)} $ if $ Z_{l}^{(r)} \not= Z_{l+1}^{(r)} $.  
\end{notation}
Assume that $ L_1, \ldots, L_k $ are all the gap numbers of $ Z^{(r)} $.
Setting $ L_0 = 0 $, we obtain the filtration of $ Z^{(r)}$:
\[
	 Z^{(r)}  = Z_{L_1}^{(r)} \supsetneq Z_{L_2}^{(r)}  \supsetneq \cdots \subsetneq Z_{L_k}^{(r)} \subsetneq Z_{L_k+1}^{(r)}  =  Z_{\infty}^{(r)}  = K^{(r)}
\] 
and 
\[
	Z_{L_{i}+1}^{(r)} = \cdots =  Z_{L_{i+1} }^{(r)}  \text{ for $ i \geqslant 0 $}.
\]

\section{Main Results}\label{sec:MainResult}
\subsection{Bigraded Structure}
 Now we are in a position to establish the relation between $ \mathcal{P} $-modules $ Z^{(r+1)}$ and $ K^{(r)} $. In the previous section, we equipped both modules with bigraded structures, defined by the degree inherited in $E^{(r)}$ and the level filtration.
Let us define the relative graded modules with respect to the level filtration as follows
\[  \gr_{\bullet, l}(Z^{(r+1)}) = Z_{l}^{(r+1)} /Z_{l+1}^{(r+1)} \]
and 
\[ \gr_{\bullet,l}(K^{(r)}) = K^{(r)}_l / K^{(r)}_{l-1} .\] 
We  show that the component $\gr_{d,l}(Z^{(r+1)})$ (degree $ d $, level $ l $)  is naturally isomorphic to the component $\gr_{d+l,l}(K^{(r)})$ in $ K^{(r)}$. 
To this end, we  construct a morphism from $ Z_l^{(r+1)} $ to $ K_{l}^{(r)}$. For $ l \geqslant 0 $, consider a $\mathcal{P}[\mu]$-homogeneous element $\sum_{I} a_I^\mu v_I $ of $\tilde{Z}^{(r+1)}_l$. From the definition of $ \tilde{Z}^{(r+1)}_l $, we have 
\begin{equation}\label{eq:DeltaKappa}
	\Delta_{r+1} \left( \sum_{\# I = r+1} a_I^\mu v_I \right)  = \sum_{\# I = r+1} a_I^\mu \mathcal{T}_I^{\mu} =  \kappa \mu^{l} \mod \mu^{l+1},
\end{equation}
for some homogeneous element $\kappa \in E^{(r)}$.
Since $ \Delta_{r+1}$ preserves degrees, it follows from \eqref{eq:DeltaKappa} that  
\begin{equation}\label{eq:degkappa}
	\deg (\sum_{\# I = r+1} a_I^\mu v_I) = \deg(\kappa \mu^{l}) = \deg\kappa - l .	
\end{equation}
Substituting $\mu =1$ to each $a_I^\mu$, we have coefficients $a_I^1 \in \mathcal{P}$ verifying  
\[
 \mh\sum_{\# I = r+1 } a_I^1 \mathcal{T}_I =  \kappa 
\]
and 
\[ 
\deg(\kappa) - \min_{\# I = r+1} \{ \mult{a_I^1} + m_I \} \leqslant l.
\]
Thus, $ \kappa $ is of level $\leqslant l$. We obtain the $\mathcal{P}$-homomorphism 
\begin{align*}
	 \tilde{\Lambda}_l : \tilde{Z}^{(r+1)}_l &\to K_l^{(r)}	\\
	 \sum_{I} a_I^\mu v_I &\mapsto \kappa:=\mh \sum_{I} a_I^{1} \mathcal{T}_I.
\end{align*}
The equality \eqref{eq:degkappa} yields that $\tilde{\Lambda}_l$ is homogeneous of degree $ l $ as a graded morphism. It is trivial to see that the kernel of $\tilde{\Lambda}_l$ equals $\Downarrow_l   \tilde{Z}^{(r+1)}_{l+1} $. For $ l\geqslant 1 $,
we have the commutative diagram: 
	\[
\begin{tikzcd}
	\tilde{Z}^{(r+1)}_{l-1} \ar[r, "\tilde{\Lambda}_{l-1}"]\ar[d, "\Uparrow_{l}"] & K_{l-1}^{(r)}\ar[d, hookrightarrow]\\
	\tilde{Z}^{(r+1)}_{l} \ar[r, "\tilde{\Lambda}_{l}"] & K_{l}^{(r)}
\end{tikzcd}\]
Combining the isomorphism \eqref{equ:Zr}, we have the map
\begin{equation}\label{eq:Lambda}
	\begin{tikzcd}
		\Lambda_l :  {Z}^{(r+1)}_{l} \ar[r,"\Downarrow_1^{-1}"]& \tilde{Z}^{(r+1)}_{l}/ \Uparrow_l \tilde{Z}^{(r+1)}_{l-1} \ar[r,"\tilde{\Lambda}_l"] &  \gr_{\bullet, l} K^{(r)} . 
	\end{tikzcd}
\end{equation}   
The following lemma implies that $ \tilde{\Lambda}_l $ ($l \geqslant 1$) in \eqref{eq:Lambda} is surjective and so is $\Lambda_l$. 
\begin{lem}\label{lem:kappa}
	If $ \kappa $ is homogeneous of level $l$ with $ l \geqslant 1 $, then 
	$\kappa = \tilde{\Lambda}_l (\theta^{\mu})$ for some homogeneous
	element $\theta^\mu \in \tilde{Z}^{(r+1)}_{l} $.  
\end{lem}
\begin{proof}
	Write 
	\[
		\kappa = \mh \sum_{\#I ={r+1} }  a_{I} \mathcal{T}_{I} 	.
	\]
	Take $  m_0 := \min_{I} \{ \mult a_{I} + m_{I}\}$ and $ a_{I}^{\mu}: = \frac{1}{\mu^{m_0 -m_{I}}} a_{I}(\mu \mathbf{x}) \in E^{(r+1)}[\mu]$. Then  
	\[
		\kappa \mu^l = \sum_{\# I = r+1} a_{I}^{\mu} \mathcal{T}_{I}^\mu	\mod \mu^{l+1}.
	\]
	It follows that  
	\[
		\theta^\mu:= \Downarrow_{l} (\sum_{\# I = r+1} a_{I}^{\mu} v_{I}) 	
	\]
	is contained in $\tilde{Z}^{(r+1)}_l$ and represents the preimage of $\kappa $. 
\end{proof} 

\begin{thm}
	For a fixed level $ l \geqslant  1 $, the kernel of $\Lambda_l :  Z^{(r+1)}_{l} \to  \gr_{\bullet, l}(K^{(r)}) $ coincides with $Z^{(r+1)}_{l+1}$. Consequently, the homomorphism
		\begin{equation}\label{equ:lambda_l}
			\lambda_l : 	\gr_{\bullet, l}(Z^{(r+1)}):= Z_{l}^{(r+1)}/Z_{l+1}^{(r+1)}  \to \gr_{\bullet, l}(K^{(r)}) 
		\end{equation}
		induced by $ \Lambda_l $ is a degree $ l $ isomorphism of graded modules. Moreover, the gap numbers of $ K^{(r)} $ are identical to those of $ Z^{(r+1)}$.
\end{thm}
\begin{proof}
	Since the kernel of $ \tilde{\Lambda}_l  $ equals $  \Downarrow_l \tilde{Z}_{l+1}^{(r+1)} $, we have 
	$ Z_{l+1}^{(r+1)} = \Downarrow_1 \ker \tilde{\Lambda}_l  \subseteq \ker \Lambda_l
$.
Thus, $ \lambda_l $ is well-defined.
	By Lemma \ref{lem:kappa}, $\Lambda_l $ is surjective and so is $ \lambda_l $.
	To determine the kernel of $ \Lambda_l $,
	we may assume that 
	\[ \theta^\mu =\sum_{\#I = r+1 }a_I^\mu v_I \in \tilde{Z}^{(r+1)}_{l} 
	\] is a homogeneous element such that $ \Lambda_l (\Downarrow_1(\theta^\mu))  $ vanishes in  $\gr_{\bullet, l}(K^{(r)})$. This yields that $\kappa :=  \tilde{\Lambda}_l(\theta^\mu) $ is contained in  $K_{l-1}^{(r)} $, i.e., the level $ l_\kappa $ of $\kappa$ is less than or equal $ l-1 $. 
	From the definition of $ \tilde{\Lambda}_l $, we get 
	\[
		\sum_{\# I = r+1 } a_I^\mu \mathcal{T}_I^\mu = \kappa \mu^{l} \mod \mu^{l+1}. 
	\]

	Case 1 : $l_\kappa = 0$. From Lemma \ref{lem:levelzero}, we know $ \kappa $ can be expressed as 
	\[
	\kappa = \sum_{\# I = r+1} b_I \mh \mathcal{T}_I
	\]
	for some homogeneous elements $b_I$ of $ \mathcal{P } $.
	Therefore, 
	\[
		0= \sum_I (a_I^\mu  -  b_I \mu^{l})  \mathcal{T}_I^\mu \mod \mu^{l+1}.
	\]
	Setting
	\[ 
		\Theta_1^\mu:= \sum_I (a_I^\mu  -  b_I \mu^{l}) v_{I} \in \tilde{Z}_{l+1}^{(r+1)}
	\]
	we have $\Downarrow_1(\theta^\mu) =\Downarrow_1 (\Theta_1^\mu) \in {Z}_{l+1}^{(r+1)}$.

	Case 2: $1 \leqslant l_\kappa \leqslant l - 1 $.
	It follows from Lemma \ref{lem:kappa} that 
	\[
		\kappa  = \tilde{\Lambda}_{l_\kappa} (\sum_I c_I^\mu v_I) 
	\]
	for some homogeneous coefficients $ c_I $ of $ \mathcal{P}[\mu]$. Equivalently,
	\[
		\kappa \cdot \mu^{l_\kappa}	= \sum_I c_I^\mu \mathcal{T}_I^\mu \mod \mu^{l_\kappa + 1}
	\] 
	so we obtain  
	\[
		0= \sum_I (a_I^\mu - \mu^{l- l_\kappa} \cdot c_I^\mu )  \mathcal{T}_I ^\mu \mod \mu^{l+1}.
	\]
	Now we set 
	\[ \Theta_2^\mu := \sum_I (a_I^\mu - \mu^{l- l_\kappa} \cdot c_I^\mu ) v_{I} .
	\]
	Then $\Theta_2^\mu \in \tilde{Z}^{(r+1)}_{l+1} $ and $\Downarrow_1(\Theta_2^\mu) = \Downarrow_1(\theta^\mu) \in Z^{(r)}_{l+1} $.
	
	In both cases, we obtain $\ker \lambda_l = Z^{(r)}_{l+1}$. So the isomorphism \eqref{equ:lambda_l} holds. 
\end{proof}
As a consequence, we have $ Z_{L}^{(r)} = Z_{\infty}^{(r)} $, where $ L $ is the maximal level verifying $ K_{L}^{(r)}= K^{(r)} $.   

 Combining the isomorphism \eqref{equ:rho1}, we have the commutative diagram of isomorphisms:
 \[
\begin{tikzcd}
	\tilde{Z}_{l}^{(r+1)} /\left \langle  \Downarrow_l(\tilde{Z}_{l+1}^{(r+1)}) , \Uparrow_l(\tilde{Z}_{l-1}^{(r+1)}) \right  \rangle   \ar[rd, "\tilde{\Lambda}_{l}"]\ar[d, "\Downarrow_{1}"] & {}\\
	\gr_{\bullet,l}(Z^{(r+1)}) \ar[r, "{\lambda}_{l}"] &  \gr_{\bullet,l}(K^{(r)})
\end{tikzcd}.
\]
% \begin{align*}
% 	 \tilde{\Lambda} &:  \tilde{Z}_{l}^{(r+1)} /\left \langle  \Downarrow_l(\tilde{Z}_{l+1}^{(r+1)}) , \Uparrow_l(\tilde{Z}_{l-1}^{(r+1)}) \right  \rangle \to K^{(r)}_l / K^{(r)}_{l-1} \\
% 			& \Lambda \left( \sum_{\# I = r+1} a_{I}^{\mu} v_{I} \right) = \mh d_{r+1} \sum_{\# I = r+1}  a_{I}^{1} v_{I}  = 
% 			\mh \sum_{\# I = r+1}  a_{I}^{1} \mathcal{T}_{I}  .
% \end{align*} 
For convenience, we write $ \tilde{\Lambda} $ for each $ \tilde{\Lambda}_l $. 
As a consequence,  $ \kappa \in K^{(r)} $ is homogeneous of level $ l $ (with $l \geqslant 1 $) if and only if there exists some element  
\[ 
	\sum_{\# I = r+1} a_{I}^{\mu} v_{I} \in  \tilde{Z}_{l}^{(r+1)} \setminus \left \langle  \Downarrow_l(\tilde{Z}_{l+1}^{(r+1)}) , \Uparrow_l(\tilde{Z}_{l-1}^{(r+1)}) \right  \rangle
\]
such that 
\[
   \kappa =\tilde{\Lambda}\left( \sum_{\# I = r+1} a_{I}^{\mu} v_{I}  \right) =\mh \left( \sum_{\# I = r+1} a_{I}^{1} \mathcal{T}_{I} \right). 
\]
According to the arguments above, we have the complete description for generators of $ K^{(r)}_l $ extending Lemma \ref{lem:levelzero}:
\begin{thm}\label{thm:generator}
	Let $ l \geqslant 1$ be an integer. For $ 1 \leqslant j \leqslant l $, assume that the elements $ \theta^{\mu}_{ j, k } \in \tilde{Z}_{j}^{(r+1)}  $ with $1 \leqslant k \leqslant k_j $  
form the generators of the quotient $\tilde{Z}_{j}^{(r+1)} /\left \langle  \Downarrow_j(\tilde{Z}_{j+1}^{(r+1)}) , \Uparrow_j(\tilde{Z}_{j-1}^{(r+1)}) \right  \rangle $. Then
the set 
\[  \left \{  \tilde{\Lambda}( \theta_{j,k}^{\mu}), \mh \mathcal{T}_{I} \right \}_{ 1 \leqslant j \leqslant l, 1 \leqslant k \leqslant k_j, \# I = r+1 }
\]
 is a homogeneous generating subset of $K^{(r)}_{l} $. 
\end{thm}

\begin{cor}\label{cor:ann}
	Suppose that $ {Z}_{l}^{(r+1)} /  {Z}_{\infty}^{(r+1)} $(resp. $ {Z}_{l}^{(r+1)} $) is generated by the single element $\Downarrow_1 ( \theta^\mu ) $ for some $\theta^\mu \in \tilde{Z}_{l}^{(r+1)}$. 
	Set $\kappa = \tilde{\Lambda}(\theta^\mu) $. Then  ${Z}_{l+1}^{(r+1)} /  {Z}_{\infty}^{(r+1)}$ (resp. ${Z}_{l+1}^{(r+1)}$) coincides with $\Ann(\kappa ; K^{(r)}_{l-1}) \cdot \Downarrow_1 ( \theta^\mu ) $.
\end{cor}
\begin{proof}
	From the isomorphism $\eqref{equ:lambda_l}$ we have 
	\[
		\frac { {Z}_{l}^{(r+1)} /  {Z}_{\infty}^{(r+1)}	} { {Z}_{l+1}^{(r+1)} /  {Z}_{\infty}^{(r+1)}}  \cong   {Z}_{l}^{(r+1)}  / { {Z}_{l+1}^{(r+1)}  }  \cong K_{l}^{(r)} /K_{l-1}^{(r)}
	\]
	which sends $\Downarrow_1 (\theta^\mu)$ to $\kappa$.
	Now the corollary follows immediately.
\end{proof}
Suppose that $ p_1, \ldots , p_s $ are generators of 
$ \Ann(\kappa ; K^{(r)}_{l-1})$, 
and $\tilde{\Lambda} ( \theta^{\mu}_{j,k})$ are 
the generators of $ K^{(r)}_{l-1}$ as in Theorem \ref{thm:generator}. We are able to give a more detailed description for Corollary \ref{cor:ann}.
One can write
\[ p_i \kappa = \sum a_{I} \mh \mathcal{T}_{I} + \sum_{j,k} b_{j,k} \tilde{\Lambda} ( \theta^{\mu}_{j,k})
\]
for homogeneous coefficients $a_{I} $, $b_{j,k} \in \mathcal{P}$. 
We set 
\[
	\Theta_i^\mu :=  p_i \theta^\mu -   \sum_I a_{I} v_{I} \mu^{l} - \sum_{j,k} b_{j,k} \theta^{\mu}_{j,k}\mu^{l-j} . 
\]
Then $\Theta_i^\mu \in \tilde{Z}^{(r+1)}_{l+1} $ and $\Downarrow_1(\Theta_i^\mu) = p_i \cdot \Downarrow_1 (\theta^\mu  )$.  In particular, the elements
\[    \Downarrow_1(\Theta_1^\mu), \ldots, \Downarrow_1(\Theta_s^\mu) 
\]
represent the generators of $ {Z}_{l}^{(r+1)} /  {Z}_{\infty}^{(r+1)} $(resp. $ {Z}_{l}^{(r+1)} $).  
\begin{thm} \label{thm:KZ}
	Let $ K^{(r)} $, $B^{(r)}$, $Z^{(r)} $ be the $\mathcal{P}$-modules defined above associated to weighted-homogeneous polynomial $ f $. Then we obtain the linear isomorphism
	\[
		K^{(r)} / B^{(r)}  \cong Z^{(r+1)} / K^{(r+1)}
	\]
	as $ \mathbb{C} $-vector spaces.
\end{thm} 
\begin{proof}
From the filtrations of $ K^{(r)}$ and $ Z^{(r+1)}$, we have the linear isomorphisms:
\[
K^{(r)} / B^{(r)} \cong   \oplus_{l=1}^{\infty} K_{l}^{(r)}/K_{l-1}^{(r)}  
\]
and
\[
Z^{(r+1)} / K^{(r+1)}  \cong  \oplus_{l=1}^{\infty} Z_{l}^{(r+1)}/Z_{l+1}^{(r+1)}.
\]
Now the theorem follows from the isomorphism \eqref{equ:lambda_l}.
\end{proof}
\begin{cor}\label{cor:H1}
	Consider the new Koszul type complex  $\Kos_*^0(f)$ of $ f $ with respect to the coordinate $ x_1, \ldots , x_n $. If $ H_{r-1}(\Kos_*^0(f)) = 0 $, then 
	\[
		H_{r}(\Kos^0_*(f)) \cong  Z^{(r+1)} / K^{(r+1)}
	\]
	as vector spaces.
	In particular, 
	\begin{equation}\label{equ:H1}
		H_{1}(\Kos^0_*(f)) \cong  Z^{(2)} / K^{(2)}.
	\end{equation}
\end{cor} 
\begin{proof}
	If $ H_{r-1}(\Kos^0_*(f)) = 0 $, then $ K^{(r-1)} = B^{(r-1)}$. It follows from Theorem  \ref{thm:KZ} that 
	$Z^{(r)} = K^{(r)}$.
	Applying Theorem \ref{thm:KZ} once again, we obtain  
	\[
		H_{r}(\Kos^0_*(f))= Z^{(r)} / B^{(r)} = K^{(r)} / B^{(r)} \cong  Z^{(r+1)} / K^{(r+1)}.
	\]
	The isomorphism \eqref{equ:H1} is deduced by the fact $ H_0( \Kos^0_*(f)) = 0 $. 
\end{proof}

\subsection{Proof of Main Theorem \ref{thm:AtMtIntro}}
In this section, we derive a formula for the Hilbert-Poincar\'{e} series of $ \gr J(f)$, which implies that formulas in Main Theorem \ref{thm:AtMtIntro}. 
\begin{notation}
	For $ l \geqslant 1  $, we let $ \mathbb{K}_{l}(t), \mathbb{Z}_{l}(t) , \mathbb{Z}_{\infty}(t), \mathbb{H}_{l}(t) $ be the Hilbert-Poincar\'{e} series of $ K_{l}^{(1)} $,  $ Z_{l}^{(2)} , Z_{\infty}^{(2)} $  and the quotient $ \gr_l(Z^{(2)}) = Z_{l}^{(2)} / Z_{l+1}^{(2)} $ respectively.
\end{notation}
From definition, we have 
\begin{equation}\label{equ:Ht}
	\mathbb{H}_{l}(t) = \mathbb{Z}_{l}(t) - \mathbb{Z}_{l+1} (t).
\end{equation}
Therefore, $\mathbb{H}_{l}(t) \not= 0 $ if and only if $ l $ is a gap number.
Since the homomorphism $ \lambda_l $ in Equation \eqref{equ:lambda_l} is of degree $  l $, we have 
\begin{equation}\label{equ:deltaK}
	\mathbb{K}_{l} - \mathbb{K}_{l-1} = t^{l} ( \mathbb{Z}_{l}(t) - \mathbb{Z}_{l+1} (t))= t^{l} \mathbb{H}_l(t).  
\end{equation}
Suppose that $ L_1 <L_2 ,\cdots < L_k $ are all the gap numbers of $ K(f)$. Then
the filtration for $Z^{(2)} $ has the form 
\[
	Z_{1}^{(2)} = \cdots = Z_{L_1}^{(2)} \supsetneq  Z_{L_2}^{(2)} \supsetneq \cdots  \supsetneq Z_{L_{k-1}}^{(2)} \supsetneq Z_{L_k}^{(2)} \supsetneq Z_{L_k+1}^{(2)} = Z_{\infty}^{(2)} 
\]
and for $1 \leqslant i \leqslant k -1 $,
\[ 
Z_{1+L_i}^{(2)} = \cdots   = Z_{L_{i+1}}^{(2)} . \]
In particular, we obtain 
\begin{equation}\label{equ:Z1t}
	\mathbb{Z}_{1}(t)  = \mathbb{Z}_{\infty}(t)  + \sum_{i=1}^{k} \mathbb{H}_{L_i} (t). 
\end{equation}
The following theorem together with formulas \eqref{equ:Mt} and \eqref{equ:At} implies Main Theorem \ref{thm:AtMtIntro}.
\begin{thm}
	Suppose that $ L_1 < L_2 <\cdots <L_k $ are all the gap numbers of $ K(f) $ (or $Z^{(2)}$). With the notation defined above, we obtain
\begin{equation}\label{equ:Jt}
	\mathbb{J}_f(t) = \frac{1}{(1-t)^n } \left ({n -\sum_{i<j}  t^{m_{i,j}}} \right ) +  \mathbb{Z}_{\infty}(t) + \sum_{i=1}^{k} (1-t^{L_i}) \mathbb{H}_{L_i} (t) .  
\end{equation}
\end{thm}
\begin{proof}
	It follows from the exact sequence
	\[
	0 \to Z^{(2)}_1 \to E^{(2)} \to K_0^{(1)} \to 0 
	\]
	that
	\[
	\mathbb{K}_{0}(t) = 	\frac{1}{(1-t)^n } \left ({\sum_{i<j}  t^{m_{i,j}}} \right ) - \mathbb{Z}_{1}(t). 
	\]
	Combining \eqref{equ:Z1t} we have 
	\[
	\mathbb{K}_{0}(t) = 	\frac{1}{(1-t)^n } \left ({\sum_{i<j}  t^{m_{i,j}}} \right ) - \mathbb{Z}_{\infty}(t)  - \sum_{i=0}^{k} \mathbb{H}_{L_i} (t).
	\]
	Since $L_k$ is the maximal gap number, we have $K_{L_k}^{(1)} = K(f) $. Hence, the Hilbert-Poincar\'{e} series $\mathbb{K}(t)$ of $ K(f)$ coincides with $\mathbb{K}_{L_k}(t)$.
	Combining with the formula \eqref{equ:deltaK}, we find   
		\begin{align}
			\mathbb{K}(t) =& \mathbb{K}_{L_k}(t) \nonumber \\
		=& \mathbb{K}_{0}(t) + \sum_{i=1}^k \left( \mathbb{K}_{L_{i}}(t) -\mathbb{K}_{L_{i-1}}(t)  \right) \nonumber \\
		=& 	\frac{1}{(1-t)^n } \left ({\sum_{i<j}  t^{m_{i,j}}} \right ) - \mathbb{Z}_{\infty}(t)  - \sum_{i=1}^{k} (1-t^{L_i})\mathbb{H}_{L_i} (t).  \label{equ:Ktnew} 
		\end{align}	 
	Now Equation \eqref{equ:Jt} is obtained by combining Equations \eqref{equ:JtKt} and \eqref{equ:Ktnew}. 
\end{proof}
In particular, if $ K(f) $ contains no gap numbers, then 
\begin{equation}\label{equ:Jt_nogap}
	\mathbb{J}_f(t) = \frac{1}{(1-t)^n } \left ({n -\sum_{i<j}  t^{m_{i,j}}} \right ) +  \mathbb{Z}_{\infty}(t) .
\end{equation}
The formula \eqref{equ:Jt} is quite explicit as both series $ \mathbb{H}_{L_i}(t) $ and $\mathbb{Z}_{\infty} (t) $ can be directly computed from the filtration of $ Z^{(2)} $. In the rest of this paper, we demonstrate some explicit calculations.

\section{Low-Dimensional Cases}\label{sec:Low}
\subsection{Two-Dimensional Case}
In this section, we assume $ n = 2 $. We verify the exactness of the modified Koszul complex and obtain explicit formulas for both $ \mathbb{A}(t) $ and $ \mathbb{M}(t) $.  

\begin{notation}
	As usual, we denote by $ f_1, f_2 $ the partial derivatives of the weighted homogeneous polynomial $ f $. Now we have quantities
	$m_1 = \mult(f_1)$ and $m_2 = \mult(f_2)$. The multiplicity of $f $ is given by 
	\[ m_0 =\mult(f) = \min\{m_1,m_2\}+1. \]
\end{notation}
	It follows from Definition \ref{def:Kf} that 
    \begin{align*}
        K(f) & = \langle \mh (a \mathcal{T}_{1,2}) \rangle_{a \in P} = \langle \mh (a) \cdot \mh \mathcal{T}_{1,2} \rangle_{a \in P} \\
        & = \langle \mh \mathcal{T}_{1,2} \rangle = B^{(1)}. 
    \end{align*}
	Therefore, $ K(f)$ has no gaps.
	The complex $ \Kos_*^{0}(f)$ in this case reduces to the exact sequence
\[
\begin{tikzcd}
	0 \ar[r] &  
	\langle \mh \mathcal{T}_{1,2} \rangle  \ar[r,"\delta_2"] & \langle v_1, v_2 \rangle \ar[r,"f_*"] & \gr J(f)   \ar[r] & 0
\end{tikzcd}
\]
In other words, the homology $H_i(\Kos^0_*(f))$ vanishes for $i = 0,1,2$.
	This yields that 
	\[
		Z^{(2)}_{\infty} \subseteq Z^{(2)} = \ker \delta_2 = 0 ,
	\]
	and then $Z^{(2)}_{\infty} = 0$.
	Substituting  
	$	\mathbb{Z}_{\infty}(t) = 0	$ and $ m_{1,2} = m_0-1 $  into Equation \eqref{equ:Jt_nogap}, we obtain the result below.
\begin{thm}
	If $n  =2$, then the Hilbert-Poincar\'{e} polynomial of $ \gr J(f) $ is given by 
	\[
		\mathbb{J}_f(t) = \frac{1}{(1-t)^2} \left( 2 - t^{m_0-1} \right).
   \]
\end{thm}
Since $ \mu_0= \dim A_{0} = \dim M_{0} = m_1 m_2$, we obtain from \eqref{equ:Mt} that 
\[
	\mathbb{M}_f(t) = \frac{m_1 m_2}{(1-t)} + \frac{  2t - t^{m_0} }{(1-t)^3}.
\]
Accordingly, from \eqref{equ:At} we obtain 
\[
	\mathbb{A}_f(t) = 	\frac{m_1 m_2}{(1-t)} + \frac{  2t-t^2 - t^{m_0} }{(1-t)^3}. 
\]
More explicitly, one may calculate the coefficients of the series to obtain the following formulas
\[
	   \tau_k =  \begin{cases}
	\frac{1}{2} k^2+\frac{3}{2}k+ m_1 m_2                                 & \text{for $k \leqslant m_0-1 $;} \\
	mk + m_1 m_2 - \frac{1}{2} (m_0-1)(m_0-2) & \text{for $ k \geqslant m_ 0 $};
\end{cases}
\]
and
\[
	\mu_k =  \begin{cases}
	  k^2+ k+ m_1 m_2                                 & \text{for $k \leqslant m_0 -1 $;} \\
	\frac{1}{2} k^2 + (m_0 -\frac{1}{2})k + m_1 m_2 - \frac{1}{2} (m_0-1)(m_0-2) & \text{for $ k \geqslant m_0 $}.
\end{cases}
\]
\subsection{Three-Dimensional Case}
Let $f $ be a weighted homogeneous singularity of embedding dimension $n = 3$.
% \[
%  \mathcal{T}_{1,2,3} =   f_1 v_{2,3} -  f_2 v_{1,3} +  f_3 v_{1,2} 
% \]
% \[ H_{0} (\Kos^0_*(f))= H_{3}(\Kos^0_*(f)) = 0 . \]
From Lemma \ref{lem:levelzero}, we get 
\[
	K^{(1)}_0 = B^{(1)}	 = \langle \mh  \mathcal{T}_{1,2} ,  \mh \mathcal{T}_{1,3}, \mh \mathcal{T}_{2,3} \rangle.
\]
Since $ H_3(\Kos^0_*(f)) = 0 $, we have $ K^{(3)} = Z^{(3)} $.
Theorems \ref{thm:Zinf} and \ref{thm:KZ} yield that  
\begin{equation*}
Z^{(2)}_{\infty} = K^{(2)} = B^{(2)} = \mathcal{P} \cdot \mh\mathcal{T}_{1,2,3} .
\end{equation*}
Therefore, we obtain 	
\begin{equation}
	\label{equ:Zinf}
	\mathbb{Z}_{\infty}(t) = \frac{t^{m_{1,2,3}}}{(1-t)^3}	.
\end{equation}
According to the equation \eqref{equ:H1} in Corollary \ref{cor:H1}, 
\[
H_{1} (\Kos^0_*(f)) \cong  Z^{(2)} / K^{(2)} = Z^{(2)} / B^{(2)} =  H_{2} (\Kos^0_*(f)). \]

As in Notation \ref{not:Theta}, denote by $ q $ the maximal factor of $\mh \mathcal{T}_{1,2,3}$. We know from Theorem \ref{thm:Theta}
that  $ Z^{(2)} = Z^{(2)}_1 $ is generated by  
	$  \Theta_q:= \frac{1}{q} \mh \mathcal{T}_{1,2,3} $. The Hilbert-Poincar\'{e} series of  $Z_{1}(t)$ gives 
	\begin{equation}\label{equ:ZThetaq}
		\mathbb{Z}_{1}(t) = \frac{t^{m_{1,2,3}-\deg(q)}}{(1-t)^3}	.
	\end{equation}
	Theorem \ref{thm:Theta} also implies that $ H_2(\Kos^0_*(f)) \cong \mathcal{P} / q $. Therefore, we obtain the following theorem.
\begin{thm}
	Let $f $ be a weighted homogeneous singularity of embedding dimension $n = 3$.
	If $q $ is the maximal common factor of $ \mh \mathcal{T}_{1,2,3} $, then 
	\[ H_{1} (\Kos^0_*(f)) \cong H_{2}(\Kos^0_*(f)) \cong  \text{underlying $\mathbb{C}$-vector space of  $ \mathcal{P} / q $}, \]
	\[ H_{0} (\Kos^0_*(f)) \cong H_{3} (\Kos^0_*(f)) = 0.
		\]
\end{thm}
\subsubsection{Case: $q$ is a constant}
We first consider the simple case where $ q $ is a constant. Then
\[ H_{2}(\Kos^0_*(f)) = H_{1}(\Kos^0_*(f))  = 0 .\] 
Applying Theorem \ref{thm:St}, we obtain the explicit formula of Hilbert-Poincar\'{e} series 
\begin{equation}\label{equ:constant}
	\mathbb{J}_f(t) = \mathbb{S}_{(m_1,m_2,m_3)} (t).
\end{equation}
Since $K^{(1)} = B^{(1)}$, we know from Lemma \ref{lem:levelzero} that $ K(f) $ has no gap numbers. 

\subsubsection{Case: $q$ is nontrivial}
Now we assume that $ q $ is the maximal nonconstant factor of $\mh \mathcal{T}_{1,2,3}$. 
To investigate the first gap number of $ K(f)$, 
we introduce the following notations:
\begin{enumerate}
	\item Let $ j $ be the maximal integer such that $ \Downarrow_k (\mathcal{T}_{1,2,3}^\mu) $ is divisible  by $ q $ for $1 \leqslant k \leqslant j$;
	\item  $\Theta_q := \frac{1}{q} \Downarrow_1 (\mathcal{T}_{1,2,3}^\mu) $ and $\Theta_q^\mu := \frac{1}{q} \Downarrow_j (\mathcal{T}_{1,2,3}^\mu) $;
	\item $\kappa_q := \tilde{\Lambda}(\Theta_q^\mu)$;
	\item $l_1 := \deg \kappa_q -\deg \Theta_q^\mu = \deg (\kappa_q) + \deg(q) - m_{1,2,3} $.
\end{enumerate}
Following the definition of $ \tilde{\Lambda}$, we write 
\[  \Delta_2(\Theta_q^\mu ) =\kappa_q \mu^{l} \mod \mu^{l+1} \]
for some integer $ l > 0 $. So we obtain $ \Theta_q =\Downarrow_1(\Theta_q^\mu)  \in Z_{l}^{(2)}$ 
 and $ \kappa_q \in K_{l}^{(1)}$. Since $ \Delta_2 $ preserves degrees, we see that 
		\[
			m_{1,2,3} - \deg(q) =\deg (\Theta_q^\mu ) = \deg(\kappa\mu^l ) = \deg(\kappa_q) - l.
		\]
This implies that $l = l_1 $. So $ \Theta_q $ is contained in $ Z^{(2)}_{l_1} $ and therefore $ Z^{(2)}_{l_1} = Z^{(2)}_1 $. 
Let $ L_1 $ be the first gap number of $ K^{(1)}$ (or $ Z^{(2)}$). Then the inclusion $ Z_{L_1+1}^{(2)} \subsetneq Z_{L_1} ^{(2)}= Z_{l_1}^{(2)} = Z_{1}^{(2)}$ implies that $ L_1 \geqslant l_1 $. This gives a lower bound for the series $ \mathbb{J}_f(t) $.

\begin{thm}\label{thm:Lower}
     Using the same notation above, we have 
	\begin{align}
		\mathbb{J}_f(t) \geqslant&  \frac{1}{(1-t)^3} \left( 3  + t^{m_{1,2,3} - \deg(q)} -t^{\deg(\kappa_q)}+ t^{\deg(\kappa_q)+ \deg(q)} - \sum_{i<j}t^{m_{i,j}} \right) \label{equ:Jneq}
		\\
		=& \mathbb{S}_{m_1,m_2,m_3}(t) +  (1-t^{l_1}) (t^{m_{1,2,3}- \deg(q)} - t^{m_{1,2,3}}) \nonumber. 
	\end{align}
	Here and after, the symbol $ \geqslant $ denotes a partial order for polynomials; namely,  
	$ \sum_{i=0}^{N} a_i t^i \geqslant \sum_{i=0}^{N} a_i t^i$
	if and only if $a_i = b_i$ for $i =0 ,\ldots, n$ and $a_{n+1} \geqslant b_{n+1}$.
\end{thm}
\begin{proof}
	Suppose that $ L_1, \ldots , L_k $ are all the gap numbers of $K(f)$.  Notice that $ L_i \geqslant l_1 $ for all $ i $.  
	It follows from Equation \eqref{equ:Jt}, we have 
	\begin{align*}
		\mathbb{J}_f(t) =& \frac{1}{(1-t)^3 } \left ({3 -\sum_{i<j}  t^{m_{i,j}}} \right ) +  \mathbb{Z}_{\infty}(t) + \sum_{i=1}^{k} (1-t^{L_i}) \mathbb{H}_{L_i} (t)   \\
		\geqslant & \frac{1}{(1-t)^3 } \left ({3 -\sum_{i<j}  t^{m_{i,j}}} \right ) +  \mathbb{Z}_{\infty}(t) + (1-t^{l_1}) \sum_{i=1}^{k}  \mathbb{H}_{L_i} (t) \\
		= & \frac{1}{(1-t)^3 } \left ({3 -\sum_{i<j}  t^{m_{i,j}}}  \right ) +  \mathbb{Z}_{\infty}(t) + (1-t^{l_1})   ( \mathbb{Z}_{1} (t) - \mathbb{Z}_{\infty} (t)),
	\end{align*}
	where the last equality deduced from Equation \eqref{equ:Z1t}.
	Substituting the equalities \eqref{equ:Zinf},\eqref{equ:ZThetaq}, we complete the proof.
\end{proof}

Notice that the equality in \eqref{equ:Jneq} can be achieved if and only if $ l_1 $ is the unique gap number of $ K(f) $.  In this case, we have $\lev(\kappa_q) = l_1 $, and 
\[  K(f) = K^{(1)}_{l_1} = \left \langle \mh \mathcal{T}_{1,2} , \mh \mathcal{T}_{1,3}\mh \mathcal{T}_{2,3}, \kappa_q \right \rangle . \]
It follows that 
\[
	Z_{l_1+1}^{(2)}  = Z_{\infty}^{(2)}= q \Theta_q \cdot \mathcal{P} . 
\]
Then the induced isomorphism 
\[
	Z^{(2)} /  Z_{\infty}^{(2)} = Z_{l_1}^{(2)} /  Z_{l_1+1}^{(2)} \to K_{l_1}^{(1)} / K_{l_1-1}^{(1)} =  K(f) / K_0^{(1)} 
\]
sends $\Theta_q$ to $\kappa_q$.
In this situation, Corollary \ref{cor:ann} yields the sufficient and necessary condition:
\begin{equation} \label{equ:Annq}
	\Ann (\kappa_q ; K_{0}^{(1)} ) = q \cdot \mathcal{P} 
\end{equation}

% \begin{thm}\label{thm:gen4}
% 	The following statements are equivalent
% 	\begin{enumerate}
% 		\item The equality in \eqref{equ:Jneq} holds.
% 		\item The number $ L_1 $ is the unique gap number of $K(f)$.
% 		\item  $\Ann (\kappa_q ; K_{0}^{(1)} ) = q \cdot \mathcal{P} $.
% 		\item $ K(f) = K^{(1)}_{L_1} = \left \langle \mh \mathcal{T}_{1,2} , \mh \mathcal{T}_{1,3}\mh \mathcal{T}_{2,3}, \kappa_q \right \rangle $.
% 		\item $
% 			Z_{L_1+1}^{(2)}  = Z_{\infty}^{(2)}(= q \Theta_q \cdot \mathcal{P} )$.
% 	\end{enumerate} 
% \end{thm}

\section{Applications: Three-Dimensional Weighted Homogeneous Singularities}\label{sec:Applications}

It is a natural and important question to characterize the homogeneous polynomials with an isolated critical point at the origin. This question has remained open for 40 years. In fact, it is the first important case of the following interesting problem. Let $X$ be a smooth projective variety in $\mathbb{CP}^{n-1} $. Then the affine cone over $X  $ in $ \mathbb{C}^n $  is an affine variety $V(f)$ for some polynomial $f $ with an isolated singularity at the origin. We naturally ask: when an affine variety with an isolated singularity at the origin is the affine cone over smooth projective variety?

Orlik and Wagreich \cite{orlik1971singularities}, together with Arnold \cite{arnol1974normal}, demonstrated that any weighted homogeneous polynomial  $ f $
  whose zero locus $V(f) $ has an isolated singularity at the origin can be deformed into one of seven canonical classes (listed in the next subsection), while preserving the differential structure of the link  $  S^5 \cap V(f) $. Later, Yau and Yu \cite{yau2005classification} extended this work by classifying three-dimensional isolated rational hypersurface singularities with $\mathbb{C}^*$-action.
\subsection{Main Results}
\begin{notation}
	We make use the following notations. Let $ a,b,c , d,e $ be non-negative integers. Set $ \sigma := d+e-1 $.
	For a subset $ I \subseteq \{a,b,c,d,e, \sigma \}$,  denote by $\underline{I}$ the minimal value of $I$. For instance, $ \underline{ab} = \min \{ a,b\} $. 
\end{notation}
We consider the hypersurface singularities defined by the following seven polynomials:
\begin{align*}
	f^{(1)} :=& x^{a+1} + y^{b+1} + z^{c+1} , \text{where $ \underline{abc} \geqslant 1 $;} \\
	f^{(2)} :=& x^{a+1} + y^{b+1} + z^c y , \text{where $ \underline{abc} \geqslant 1 $;} \\
	f^{(3)}:=& x^{a+1} + y^{b} z  + z^c y, \text{ where  $ a \geqslant 1 $ and $ 2 \leqslant b \leqslant c$};  \\
	f^{(4)}:=& x^{a+1} + y^b z + z^c x,\text{where $ \underline{abc} \geqslant 1 $}; \\
	f^{(5)}:=& x^a y + y^b z +z^c x,\text{ where  $ 1\leqslant a \leqslant \underline{bc} $ and $ \underline{bc} \geqslant 2$}; \\
	f^{(6)}:=& x^{a+1} + x y^b + x z^c + y^{d} z^{e},\text{ where }  1 \leqslant b \leqslant c, a (be + cd )= (a+1)bc;\\ 
	f^{(7)} :=& x^a y + x y^b + xz^{c}+ y^{d} z^{e},  \text{ where }  (a-1) (be + cd )= (ab-1)c, \underline{ab}\geqslant 2.
\end{align*}

\begin{thm}\label{thm:homology}
	The homology groups of $ \Kos_*^{0}(f^{(1)}) $ always vanish. For $ i = 2, \ldots, 7 $, the homology groups of $ \Kos_*^{0}(f^{(i)}) $ are nontrivial if and only if the coefficients of $f^{(i)}$ verify the corresponding conditions listed below. In addition,  
	\[ H_1(\Kos_*^{0}(f^{(i)}) ) \cong H_2(\Kos_*^{0}(f^{(i)}) ) \cong \mathcal{P} / q, \] 
	where $q$ is the factor listed in each case. \newline
Case $f^{(2)}$:
		\begin{enumerate} 
			\item $ b < c < a$, where $q = y $;
			\item $ 2 \leqslant c <\underline{ab} $, where $q = z^{c-1}$.
		\end{enumerate}
Case $f^{(3)}$:
		\begin{enumerate} 
			\item $ 2 \leqslant b < \underline{ac} $, where $q = y^{b-1} $. 
		\end{enumerate}
Case $f^{(4)}$:
		\begin{enumerate} 
			\item	$2 \leqslant b < \underline{ac} $, where $ q = y^{b-1} $; 
			\item  $ a < c <b $, where $q = x $;
			\item $ 2 \leqslant c< \underline{ab} $, where $ q = z^{c-1} $.  
		\end{enumerate}
Case $f^{(5)}$:
		\begin{enumerate} 
			\item $ 2 \leqslant a < \underline{bc} $, where $ q = x^{a-1}$.
		\end{enumerate}
Case $f^{(6)}$:
		\begin{enumerate} 
			\item $ a < b <  \sigma $, where $q =x $; 
		\item $ d\geqslant 2, e \geqslant 1$, $  b = \sigma < \underline{ac} $, where $ q = y^{d-1} $;
		\item $ d \geqslant 0, e \geqslant 1$, $ 2 \leqslant  b < \underline{ a c\sigma} $, where $ q = y^{b-1 } $.
		\end{enumerate}
Case $f^{(7)}$:
		\begin{enumerate}
			\item $a < \underline{bc} $, where $q = x^{a-1}$;
			\item $  \underline{de \sigma} \geqslant 1  $, $ \sigma < b < \underline{ac} $,  where $ q = y^{d-1} z^{e-1} $;
			\item $ d\geqslant 2$, $ e \geqslant 1$, $ \sigma = b < \underline{ac} $, where $ q = y^{d-1}   $;
			\item  $ d\geqslant 0$, $ e \geqslant 1$, $ 2 \leqslant b < \underline{ac\sigma} $,  where $ q = y^{b-1}   $;
			\item $  \underline{de \sigma} \geqslant 1  $, $ \sigma < c < \underline{ab} $,  where $ q = y^{d-1} z^{e-1}  $;
			\item $ d\geqslant 1$, $ e \geqslant 2$, 
			$ \sigma = c < \underline{ab} $, where $ q =z^{e-1}  $;
			\item $ d\geqslant 1$, $ e \geqslant 0$, $ 2 \leqslant c < \underline{ab\sigma} $, where $ q =  z^{c-1}  $.
		\end{enumerate}

\end{thm}

\begin{notation}
	Denote the symbols:
\begin{align*}
	  \tilde{a} &= \min \{ 2 a, b+c\},
	 & \tilde{b} &= \min \{ 2 b, a+c\} ,\\
	  \tilde{c} &= \min \{ 2c, a+b \} , 
	 & \tilde{\sigma} &=  \min \{ 2 \sigma  , c + b \} \\
	  \tilde{c}_2 &= \min \{ \sigma +a, 2c \},
	 & \tilde{b}_2 &= \min \{ 2b,\sigma+a \},\\
	  \tilde{b}_3 &= \min \{ 2b, \sigma +a , c+ a \} ,
	 & \tilde{\sigma}_3  &= \min \{ a+b , c + \sigma  \}.
\end{align*} 
\end{notation}
\begin{thm}\label{thm:Main}
	The Hilbert-Poincar\'{e} series $\mathbb{J}_{f^{(i)}}(t)$ of $ \gr J(f^{(i)}) $ is given by 
	\[
		\mathbb{J}_{f^{(i)}}(t) = \frac{3 + \mathbb{L}_i(t)}{(1-t)^3},
	\]
	where $ \mathbb{L}_i(t) $'s are the polynomials listed as follows.
\end{thm}
\[ \mathbb{L}_1(t)= -t^{\underline{ab}} - t^{\underline{ac}}- t^{\underline{bc}} + t^{\underline{ab}+\underline{ac}+\underline{bc}-\underline{abc} } .\]
\[
	\mathbb{L}_2(t)= \begin{cases} - 2 t^{b} -t^{c}   + t^{b+c} +t^{b-1+c}-t^{b+c} -t^{ \tilde{c} -1} + t^{ \tilde{c} }  & \text{for $ b \leqslant c < a$;}\\
		-3 t^c + t^{2c} +   t^{1+c}-t^{2c}   -t^{\underline{ab}+1} +t^{\underline{ab}+c}   & \text{for $  c <\underline{ab}  $;}\\ 
		-t^{\underline{ab} } -t^{a} -t^{\underline{bc}} + t^{\underline{bc}+a}     & \text{for $  a \leqslant c $.}
	\end{cases}
	\]
\[
	\mathbb{L}_3(t) =  \begin{cases} -3 t^{b} + t^{b+1} -t^{\underline{ac}+1}+t^{\underline{ac}+b-1} 	-t^{\tilde{c}-1} +t^{\tilde{c}}    & \text{for $ 2 \leqslant b <a $;}\\
		-2 t^a -t^b+ t^{a+b}  & \text{for $ b \geqslant a $} .
	\end{cases}	
	\]
	\[
	\mathbb{L}_4(t) =  \begin{cases} 
		-2 t^a-t^{\underline{bc}} -t^{a+ \underline{bc}} +t^{a-1+c} - t^{a+c} - t^{\tilde{c}+c-\underline{bc}-1} + t^{\tilde{c}+c- \underline{bc}}    & \text{for $a \leqslant \underline{bc} $;}\\
		-3 t^b + t^{b+1} - t^{\underline{ac}+1} + t^{\underline{ac}+b}   & \text{for $b \leqslant \underline{ac}$;} \\
		- 3t^c+t^{c+1} -t^{\underline{ab}}(t - 2 t^{c}+t^{c+1} ) -t^{\tilde{b}}+t^{\tilde{b}+1}  & \text{for $ 2 \leqslant c \leqslant \underline{ab} $;} \\
		- 3t+t^{2}  & \text{for $ c = 1 $}.  
	\end{cases}	
	\]
	\[
	\mathbb{L}_5(t) =  \begin{cases} 
		- 3 t^a  +t^{a+1}  -t^{\underline{bc}}  (t - 2 t^{a} + t^{a+1} ) -t^{\tilde{c}} (1 - t)   & \text{for $ 2 \leqslant a \leqslant \underline{bc}$}; \\
		- 3t + t^2  & \text{for $a = 1 $}.
	\end{cases}	
	\] 
	 \[
		\mathbb{L}_6(t) = \begin{cases}
			-2t^a-t^b+t^{a+b-1}-t^{\min \{ 2b, \sigma +a \}}(t^{-1}-1)  & \text{for $ a \leqslant b <  \sigma $;} \\
			-3t^b+t^{2b} +(1-t^{\underline{ac}-b})(t^{b+e}-t^{2b})   & \text{for $ b = \sigma $ }; \\
			\substack{-3 t^b + t^{b+1}   -t^{\underline{ac\sigma} }(t - 2t^{b-d+1} + t^{b-d+2})-t^{\underline{ac}}(t^{2-d+1} -t^{b}) \\ -t^{ \min\{2 \sigma -b,a,c \} } (t^{b-d+1} - t^{b-d+2} )   } & \text{for $ 2 \leqslant b <  \underline{acd\sigma} $;} \\
			-3 t^b + t^{b+1}  -t^{\underline{ac}} (t  -t^{b}) & \text{for $ b = 1 $ or $b = d < \underline{ac \sigma}  $;} \\
			-t^a -t^{\underline{b\sigma}}-t^{\underline{c\sigma}} + t^{a+\underline{b\sigma}}  & \text{for $ a = b $ or $ b\geqslant \sigma, a+1 $}. 
		\end{cases}
	\]

	The precise expression of $ \mathbb{L}_7(t) $ is far more complicated. We split it into following cases:

	Case $ b< \underline{ac\sigma}$:
	$ \mathbb{L}_7(t) = - 3t^{b} + t^{2b} + \mathbb{B}(t) 
	$, where $ \mathbb{B}(t)$ equals  
	\[
		\begin{cases}
			(1-t^{a-b})(t^{b+1} - t^{2b-1}) + (1-t^{\tilde{a}-2b})(t^{2b-1} - t^{2b} ) & \text{for  $ d = 0 , b $;} \\
			 (1-t^{\sigma-b})(t^{b+1}-2t^{2b}+t^{2b+1})+(t^{2b}-t^{\tilde{\sigma}})(1 - t) & \text{for $ d = 1 $}; \\
			  (1-t^{a-b})(t^{b+1}-t^{2b-1}) +(1-t^{\tilde a -2 b})(t^{2b-1} - t^{2b})  & \text{for $ 2 \leqslant d \leqslant b-1 $,  and $ a \leqslant \sigma $.}
		\end{cases}
	\]
	For the case $ 2 \leqslant d \leqslant b-1 $,  and $ a > \sigma $, we only know that
	\[ 
	\mathbb{B}(t) \geqslant  (1-t^{\sigma-b})(t^{b+1} - 2t^{b-d}+t^{b-d+1}) +(t^b-t^{\min \{ 2 \sigma - b ,a  ,c  \}})(2t^{-d} - t^{-d+1} - t^{b}) .
	\]

	Case $ b \geqslant \underline{ac\sigma}$: In this case, $\mathbb{L}_7(t) $ is given by  
	\[
		\begin{cases}
			-3t^a + t^{a+1} -t^{c} (t - 2t^{a}+ t^{a+1})  -t^{\tilde{c}_2} (1-t)  & \text{for $a < c \leqslant b$}; \\
			-3t^a + t^{a+1}  -t^{b} ( t - t^{a-1} ) -t^{\tilde{b}_3} ( t^{-1}-2  + t ) -t^{\tilde{b}_2} ( 1-t)    & \text{for $a < b < c$}; \\
			-3t^\sigma + t^{ \sigma+1}  -t^{b} (t -t^{e})   -t^{\underline{ac}} (t^{e } - 2t^{\sigma} + t^{\sigma+1}) -t^{\tilde{\sigma}_3 } (1 -t ) & \text{for $\sigma < b < \underline{ac}$};\\
			- 3t^{b} + t^{b+e} -t^{\underline{ac}}(t^{e} - t^{b}) & \text{for $ \sigma =  b < \underline{ac}$;} \\
			- 3t^\sigma + 3 t^{2 \sigma } + (1- t^{c- b})(t^{\sigma+1 - t^{\sigma+d}}) +(1-t^{\underline{ab} - \sigma}) (t^{\sigma +d } - t^{2\sigma})  & \text{for $ \underline{de \sigma } \geqslant 1$, $ \sigma< c < \underline{ab} $;} \\
			- 3 t^c + t^{2c} +(1 - t^{\underline{ab} -c })(t^{c+d} - t^{2c}) & \text{for $ d \geqslant 1 , e \geqslant 2$,}\\
			 & \text{ and $\sigma = c < \underline{ab} $.}
		\end{cases}
\]
When $ d \geqslant 1 $, $2 \leqslant c < \underline{ab\sigma}$, there exists an inequality
\[
	\mathbb{L}_7(t) \geqslant  \frac{1}{(1-t)^3} \left( 3 - 3 t^c + t^{2c} +(1 - t^{a -c })(t^{c+1} - t^{2c}) \right) 	
\]
and the equality holds for $ e = 0 $ or $c $.

Applying Theorem \ref{thm:Main} and Equations \eqref{equ:Mt} and \eqref{equ:At}, we obtain Main Theorem \ref{thm:seriesIntro}. 

\subsection{Proof of the Case $f= f^{(5)} $}
Since the complete proof of Theorem \ref{thm:Main} is lengthy, we demonstrate the case when $f= f^{(5)} $. For convenience, we use the symbol
	\[ 
		\Xi_{a}^{b} = \begin{cases}
			1 & \text{ when $ a \leqslant b $};\\
			0 & \text{ when $ a > b $}.\\
		\end{cases}
	\]  
\subsubsection{Homology}
By definition, we have 
	\[
	f  = x^{a} y + y^{b} z +z^{c} x .	
	\]
	Without loss of generality, we assume that $ a \leqslant b $ and $ a \leqslant c $ and $\underline{bc} \geqslant 2$. 
	The element $\mathcal{T}_{1,2,3} \in Z_1^{(2)}$ is defined as 
	\[  \mathcal{T}_{1,2,3} =  (a x^{a-1} y + z^{c}) v_{2,3 } -  (x^{a} + b y^{b-1} z   ) v_{1,3} + ( y^{b} + c z^{c-1} x )v_{1,2}.
\]
	\begin{prop}
		The element $\mh \mathcal{T}_{1,2,3}$ is divisible if and only if 
		$ 2 \leqslant a < \underline{bc}$. In this case, the maximal factor of $ \mh \mathcal{T}_{1,2,3} $ equals $q = x^{a-1} $. 
	\end{prop}
	\begin{proof}
		Case $ a = c $. 
		Then 
		\[  \mh \mathcal{T}_{1,2,3} =  (a x^{a-1} y + z^{a}) v_{2,3 } -  (x^{a} + \Xi_{a}^{b} b y^{b-1} z   ) v_{1,3} + ( \Xi_{b}^{a} y^{b} + c z^{a-1} x )v_{1,2}.
		\]
		Therefore, $ \mh \mathcal{T}_{1,2,3} $ is not divisible. 

		Case $ a = b $.
Then 
		\[  \mh \mathcal{T}_{1,2,3} = (a x^{a-1} y + \Xi_{c}^a z^{c}) v_{2,3 } -  (x^{a} + a y^{a-1} z   ) v_{1,3} + ( y^{a} + \Xi_{c}^a c z^{c-1} x )v_{1,2}.
		\]
		In this case $ \mh \mathcal{T}_{1,2,3} $ is also not divisible.

		Case $ a <  \underline{bc}$.
		Observe that 
		\[  \mh \mathcal{T}_{1,2,3} =  (a x^{a-1} y ) v_{2,3 } -  x^{a}    v_{1,3}  
		\]
		is divisible by $q = x^{a-1} $. 
		
	\end{proof}

	\subsubsection{Case $2 \leqslant a< \underline{bc}$}
	Suppose that $2 \leqslant  a < \underline{bc}$. 
	Recall 
	\begin{align*}
		\mathcal{T}_{1,2}^{\mu} &= (ax^{a-1} y  v_2 -x^a v_1 )  - (b y^{b-1} z   v_1) \mu^{b-a} + (z^c v_2) \mu^{c-a} 	\\
		\mathcal{T}_{1, 3 }^{\mu}   &=  ax ^{a-1} y v_3  + (z^c v_3  - c z^{c-1} x v_1 ) \mu^{c-a} - ( y^b v_1 ) \mu^{b-a}  \\
		\mathcal{T} _{2, 3 }^{\mu} &=  x^a  v_3+ (b y^{b-1} z v_3 - y^b v_2 ) \mu^{b-a}   - (c z^{c-1} x )v_2	\mu^{c-a}.
	\end{align*}
	Now we have 
	\[
		\Theta_q = a y v_{2,3} - x v_{1,3}  
	\]
	and 
	$ q = x^{a -1 } $. Applying the differential $ \Delta_2 $, we get 
	\begin{equation}\label{equ:Delta}
	\begin{aligned}
			\Delta_2 (\Theta_q ) = & 
		   -x  \mathcal{T}_{1, 3 }^{\mu}  + a y \mathcal{T} _{2, 3 }^{\mu}  \\ 
		 =&   - (x z^c v_3  - c z^{c-1} x^2 v_1 ) \mu^{c-a} + ( x y^b v_1 ) \mu^{b-a } \\  
		 &+ ay \left( (b y^{b-1} z v_3 - y^b v_2 ) \mu^{b-a}   + (c z^{c-1} x )v_2	\mu^{c-a} \right)\\  
		 = & y^b ( x  v_1 - a y  v_2 + a b  z v_3 ) \mu^{b-a}  - z^{c-1} x ( - c x v_1  + a c y v_2 + z v_3 ) \mu^{c-a}.  
	  \end{aligned}
	\end{equation}
  It follows from the definitions that 
  \[ 
	Z_{\infty} =  x^{a-1} \Theta_q \cdot \mathcal{P}
\]
and 
\[
K_0^{(1)} = \langle x^{a-1}  (a y  v_2 -x v_1) , ax ^{a-1} y v_3,  x^a  v_3 \rangle .
\]

\subsubsection{Case: $a < b \leqslant c$}
From the expression \eqref{equ:Delta} of $ \Delta_{2}(\Theta_q) $, we have 
\begin{align*}
	\tilde{\Lambda} (\Theta_q ) =   &  y^b ( x  v_1 - a y  v_2 + a b  z v_3 )  - \Xi_{c}^{b} z^{c-1} x ( - c x v_1  + a c y v_2 + z v_3 )  .
\end{align*}
It can be checked that 
	 \[
  		\Ann ( \tilde{\Lambda} (\Theta_q ), K^{(1)}_0 ) = x^{a-1} \cdot  \mathcal{P}. 
	\] 
  Therefore, the equality of Theorem \ref{thm:Lower} holds, i.e., 
  \[
		\mathbb{J}_f(t) = \frac{1}{(1-t)^3} \left(  3 -3 t ^a + t^{2a } + (1-t^{b-a}) (t^{a+1} -t^{2 a } )   \right)  	.  
  \]
  Notice that $ l_1 = b -a $ is the unique gap number. 
\subsubsection{Case: $ a < c < b $}
\begin{lem}\label{lem:computation}
	Assume that $ a < c < b $, $ L_1:= c-a$ and $ L_2 := \min\{ b-a, 2(c-a) \} $. Then we obtain the filtration
	\begin{align*}
		Z_1 =& Z_{2}= \cdots = Z_{L_1} =  \Theta_q \cdot \mathcal{P}	 \\
		\subsetneq &  Z_{L_1+1} = \cdots = Z_{L_2} =   
		\langle x^{a-1} \Theta_q , x^{a-2} y \Theta_q \rangle  \\
		\subsetneq &  Z_{L_2+1} = \cdots = Z_{\infty} = x^{a-1}\Theta_q \cdot \mathcal{P}	
	\end{align*}
\end{lem}
	 \begin{proof}
		For this case we have $ L_1 = c-a $ and then 
		\[
			\tilde{\Lambda} (\Theta_q) = 	-  z^{c-1} x ( - c x v_1  + a c y v_2 + z v_3 ) , 
		\]
		  \[
		\Ann(\tilde{\Lambda} (\Theta_q); K^{(1)}_{0}) = \langle x^{a-1} , y x^{a-2} \rangle ,   
		  \]
		  and 
		  \[
			  Z^{(2)}_{L_1+1} = y x^{a-2}  \Theta_q \cdot \mathcal{P} \mod Z^{(2)}_{\infty}.
		  \]
		  The relations of $\tilde{\Lambda} (\Theta_q) $ and $ K_0^{(1)} $ are generated by
		  \begin{align*}
			a y x^{a-2} \tilde{\Lambda} (\Theta_q) + z^{c} \mh (\mathcal{T}_{1,3}) + ac z^{c-1} y \mh (\mathcal{T}_{1,2})  &= 0.
		 \end{align*}
		 We define 
		 \[
		 \Theta' = a y x^{a-2}  \Theta_q + \mu^{c-a} (z^{c} v_{1,3} + ac z^{c-1} y v_{1,2})
		\]
		 Therefore, $\Theta' \in \tilde{Z}_{L_1+1}^{(2)}$ such that $\Downarrow_1\Theta'$ is the generator of $Z_{L_1+1}^{(2)}/Z_{\infty}^{(2)}$.
		   
		 Now we consider 
		 \begin{align*}
			\Delta_2 (\Theta') :=  &  a x^{a-2} y^{b+1} ( x  v_1 - a y  v_2 + a b  z v_3 )\mu^{b-a}  \\
			&-  z^{2c-1}   ( c x v_1 - ac y v_2 - z v_3   ) \mu^{2c-2a} - (1+abc) y^b z^c v_1  \mu^{b+c-2a}    .
		 \end{align*}	 
	We have $\Theta' \in \tilde{Z}_{L_1+1}^{(2)}$ where $L_2 = \min \{ 2c -2a , b-a \}$.  It yields that 
	\begin{align*}
		\tilde{\Lambda} (\Theta') = &\Xi_{2c}^{b+a} \cdot a x^{a-2} y^{b+1} ( x  v_1 - a y  v_2 + a b  z v_3 ) \\
		&+ \Xi_{b+a}{2c}  \cdot z^{2c-1}   ( -c x v_1 + ac y v_2 + z v_3   ).
	\end{align*}
	Notice that 
	\[
K_{L_1}^{(1)} = \langle x^{a-1}  (a y  v_2 -x v_1) , ax ^{a-1} y v_3,  x^a  v_3,   z^{c-1} x (  c x v_1  - a c y v_2 - z v_3 )  
 \rangle .
\]
	We have
	\[
	\Ann(\tilde{\Lambda} (\Theta') ; K^{(1)}_{L_1} ) = x. 
	\]
	 Thus,
	 \[
		Z_{L_2+1} / Z_{\infty} = [x \cdot \Theta'] = 0. 	 
	 \] 
	 That is 
	 \[
		Z_{L_2+1} = Z_{\infty} . 	 
	 \]
\end{proof}
It follows from Lemma \ref{lem:computation} that 
\begin{align*}
	\mathbb{H}^{L_1} =& \frac{1}{(1-t)^3} \left( t^{a+1} - 2 t^{2a} + t^{2a+1} \right) , \\
	\mathbb{H}^{L_2} =&  \frac{1}{(1-t)^3} \left(  t^{2a} - t^{2a+1}  \right) , \\
	\mathbb{Z}^{\infty} =& \frac{t^{2a}}{(1-t)^3}    .
\end{align*}
Substituting these into the formula \eqref{equ:Jt}, we obtain
\begin{align*}
	\mathbb{J}_f(t)  =&  \frac{1}{(1-t)^3} \left( 3 - 3t^a + t^{2a}+ (1-t^{L_1}) (t^{a+1} - 2 t^{2a} + t^{2a+1})  + (1-t^{L_2})  (  t^{2a} - t^{2a+1})  \right)\\
	= & \frac{1}{(1-t)^3} ( 3 + \mathbb{L}_{5}(t) ).
\end{align*}

This completes the proof for $\mathbb{L}_5(t)$.

% \end{thebibliography}
\bibliographystyle{amsplain}
\bibliography{papers}

\providecommand{\bysame}{\leavevmode\hbox to3em{\hrulefill}\thinspace}
\providecommand{\MR}{\relax\ifhmode\unskip\space\fi MR }
\providecommand{\MRhref}[2]{%
  \href{http://www.ams.org/mathscinet-getitem?mr=#1}{#2}
}
\providecommand{\href}[2]{#2}
\begin{thebibliography}{10}

\bibitem{arnol1974normal}
Vladimir~I Arnol'd, \emph{Normal forms of functions in neighbourhoods of
  degenerate critical points}, Russian mathematical surveys \textbf{29} (1974),
  no.~2, 10.

\bibitem{brieskorn1966beispiele}
Egbert Brieskorn, \emph{Beispiele zur differentialtopologie von
  singularit{\"a}ten}, Inventiones mathematicae \textbf{2} (1966), no.~1,
  1--14.

\bibitem{greuel2007introduction}
Gert-Martin Greuel, Christoph Lossen, and Eugenii~I Shustin, \emph{Introduction
  to singularities and deformations}, Springer Science \& Business Media, 2007.

\bibitem{GreuelPham2017}
Gert-Martin Greuel and Thuy~Huong Pham, \emph{Mather-yau theorem in positive
  characteristic}, Journal of Algebraic Geometry \textbf{26} (2017), 347--355.

\bibitem{hirsch2006deformations}
Tobias Hirsch and Bernd Martin, \emph{Deformations with section: cotangent
  cohomology, flatness conditions, and modular subgerms}, Journal of
  Mathematical Sciences \textbf{132} (2006), 739--756.

\bibitem{hussain2023k}
Naveed Hussain, Zhiwen Liu, Stephen~S.-T. Yau, and Huaiqing Zuo, \emph{k-th
  milnor numbers and k-th tjurina numbers of weighted homogeneous
  singularities}, Geometriae Dedicata \textbf{217} (2023), no.~2, 34.

\bibitem{lin2004classification}
K-P Lin and Stephen~S.-T. Yau, \emph{Classification of affine varieties being
  cones over nonsingular projective varieties: Hypersurface case},
  Communications in analysis and geometry \textbf{12} (2004), no.~5,
  1201--1219.

\bibitem{mather1982classification}
John~N Mather and Stephen~S.-T. Yau, \emph{Classification of isolated
  hypersurface singularities by their moduli algebras}, Inventiones
  mathematicae \textbf{69} (1982), no.~2, 243--251.

\bibitem{milnor2016singular}
John Milnor, \emph{Singular points of complex hypersurfaces (am-61), volume
  61}, vol.~61, Princeton University Press, 2016.

\bibitem{Milnor1970}
John Milnor and Peter Orlik, \emph{Isolated singularities defined by weighted
  homogeneous polynomials}, Topology \textbf{9} (1970), no.~4, 385--393.

\bibitem{orlik1971singularities}
Peter Orlik and Philip Wagreich, \emph{Singularities of algebraic surfaces with
  $\mathbb{C}^*$ action}, Mathematische Annalen \textbf{193} (1971), no.~2,
  121--135.

\bibitem{pham1965formules}
Fr{\'e}d{\'e}ric Pham, \emph{Formules de picard-lefschetz
  g{\'e}n{\'e}ralis{\'e}es et ramification des int{\'e}grales}, Bulletin de la
  Soci{\'e}t{\'e} Math{\'e}matique de France \textbf{93} (1965), 333--367.

\bibitem{saito1971quasihomogene}
Kyoji Saito, \emph{Quasihomogene isolierte singularit{\"a}ten von
  hyperfl{\"a}chen}, Inventiones mathematicae \textbf{14} (1971), 123--142.

\bibitem{seeley1990variation}
Craig Seeley and Stephen~S.-T. Yau, \emph{Variation of complex structures and
  variation of lie algebras}, Inventiones mathematicae \textbf{99} (1990),
  545--565.

\bibitem{xu1993durfee}
Yi-Jing Xu and Stephen~S.-T. Yau, \emph{Durfee conjecture and coordinate free
  characterization of homogeneous singularities}, Journal of Differential
  Geometry \textbf{37} (1993), no.~2, 375--396.

\bibitem{yau1983continuous}
Stephen~S.-T. Yau, \emph{Continuous family of finite-dimensional
  representations of a solvable lie algebra arising from singularities},
  Proceedings of the National Academy of Sciences \textbf{80} (1983), no.~24,
  7694--7696.

\bibitem{yau1991solvability}
\bysame, \emph{Solvability of lie algebras arising from isolated singularities
  and nonisolatedness of singularities defined by $sl(2,\mathbb{C})$ invariant
  polynomials}, American Journal of Mathematics (1991), 773--778.

\bibitem{yau2005classification}
Stephen~S.-T. Yau and Yung Yu, \emph{Classification of 3-dimensional isolated
  rational hypersurface singularities with $\mathbb{C}^*$-action}, The Rocky
  Mountain Journal of Mathematics \textbf{35} (2005), no.~5, 1795--1809.

\end{thebibliography}

\end{document}